\documentclass[11pt,a4paper]{article}

\usepackage{epsf,epsfig,amsfonts,amsgen,amsmath,amstext,amsbsy,amsopn,amsthm
}
\usepackage{amsmath,times,mathptmx}
\usepackage{amsfonts,amsthm,amssymb}
\usepackage{amsfonts}
\usepackage{graphics}
\usepackage{latexsym,bm}
\usepackage{amsfonts,amsthm,amssymb,bbding}
\usepackage{indentfirst}
\usepackage{graphicx}
\usepackage{color}
\usepackage[colorlinks=true,anchorcolor=blue,filecolor=blue,linkcolor=blue,urlcolor=blue,citecolor=blue]{hyperref}
\usepackage{float}
\usepackage{tikz}
\newtheorem{thm}{Theorem}

\newtheorem{lemma}{Lemma}
\newtheorem{false statement}{False statement}

\theoremstyle{definition}

\newtheorem{claim}{Claim}

\newtheorem{remark}[claim]{Remark}

\newtheorem{problem}{Problem}

\begin{document}

\title{Eigenvalues of signed graphs
\footnote{Supported by NSFC (Nos. 11901498, 11771141 and 12011530064), Tianshan Youth Project of Xinjiang(No. 2019Q069), the Scientific Research Plan of Universities in Xinjiang, China (No. XJEDU2021I001), XJTCDP (No. 04231200746), BS (No. 62031224601).}}
\author{Dan Li$^a$, Huiqiu Lin$^{b}$\thanks{Corresponding author. E-mail: huiqiulin@126.com (H.Q. Lin), ldxjedu@163.com (D. Li), mjx@xju.edu.cn (J.X. Meng).}, Jixiang Meng$^a$\\
{\footnotesize $^a$College of Mathematics and System Science, Xinjiang University, Urumqi 830046, PR China}\\
{\footnotesize $^b$School of Mathematics, East China University of Science and Technology, Shanghai 200237, PR China}}
\date{}

\maketitle {\flushleft\large\bf Abstract:}
Signed graphs have their edges labeled either as positive or
negative. $\rho(M)$ denote the $M$-spectral radius of $\Sigma$, where $M=M(\Sigma)$ is a real symmetric graph matrix of $\Sigma$. Obviously, $\rho(M)=\mbox{max}\{\lambda_1(M),-\lambda_n(M)\}$.
Let $A(\Sigma)$ be the adjacency matrix of $\Sigma$ and $(K_n,H^-)$ be a signed complete graph whose negative edges induce a subgraph $H$. In this paper, we first focus on a central problem in spectral extremal graph theory as follows:
Which signed graph with maximum $\rho(A(\Sigma))$ among $(K_n,T^-)$ where $T$ is a spanning tree?
To answer the problem, we characterize the extremal signed graph with maximum $\lambda_1(A(\Sigma))$ and minimum $\lambda_n(A(\Sigma))$ among $(K_n,T^-)$, respectively.
Another interesting graph matrix of a signed graph is distance matrix, i.e. $D(\Sigma)$ which was defined by Hameed, Shijin, Soorya, Germina and Zaslavsky \cite{S.K. Hameed}. Note that $A(\Sigma)=D(\Sigma)$ when $\Sigma\in (K_n,T^-)$. In this paper, we give upper bounds on the least distance eigenvalue of a signed graph $\Sigma$ with diameter at least 2. This result implies a result proved by Lin \cite{H.Q. Lin-1} was originally conjectured by Aouchiche and Hansen \cite{M. Aouchiche}.

\vspace{0.1cm}
\begin{flushleft}
\end{flushleft}

\section{Introduction}
All the underlying graphs in our consideration are simple and connected, unless otherwise stated. A signed graph $\Sigma=(G,\sigma)$ consists of a underlying graph $G=(V,E)$ with a signature function $\sigma:E\rightarrow\{-1,1\}$.
The (unsigned) graph $G$ is said to be the underlying graph of $\Sigma$, while the function $\sigma$ is called the signature of $\Sigma$. In signed graphs, edge signs are usually interpreted as $\pm1$. Signed graphs first appeared in works of Harary \cite{F. Harary} and Cartwright and
Harary \cite{Cartwright D}, and the matroids of graphs were extended to matroids of signed graphs
by Zaslavsky \cite{T. Zaslavsky}. Chaiken \cite{S. Chaiken} and Zaslavsky \cite{T. Zaslavsky} obtained the Matrix-Tree Theorem for signed graph independently. In fact, the theory of signed graphs is a special case of that of gain graphs and of biased graphs \cite{T. Zaslavsky1}. In the very beginnings, these graphs are studied in the context of social psychology where the vertices are considered as individuals, while positive edges represent friendships and negative edges enmities between them. The notion of balance was introduced by Harary in \cite{F. Harary}, it plays a central role in
the matroid theory of signed graphs. A signed cycle is called positive (resp. negative) if it contains an even (resp. odd) number of negative edges. A signed graph is balanced if all its cycles are positive; otherwise it is unbalanced. Unsigned graphs are treated as (balanced) signed graphs where all edges get a positive sign, that is, the all-positive signature.

Let $M=M(\Sigma)$ be a real symmetric graph matrix of a signed graph $\Sigma=(G,\sigma)$ and $P_{M}(\lambda)=\mbox{det}(\lambda I-M)$ be the $M$-polynomial. The spectrum of $M$ is called the $M$-spectrum of the signed graph $\Sigma$. As usual, we use
$\lambda_1(M) \geq\lambda_2(M) \geq\cdots  \geq\lambda_n(M)$
to denote the spectrum of $M$. The adjacency matrix of $\Sigma$ is defined as $A(\Sigma)=(a^{\sigma}_{ij})$, where $a^{\sigma}_{ij}=\sigma(v_iv_j)$ if $v_i \thicksim v_j$,  and $0$ otherwise.

Our first motivation is from Koledin and Stani\'{c} \cite{T. Koledin} who studied the connected signed graphs of fixed order, size, and number of negative edges with maximum index. In the paper, they conjectured that if $\Sigma$ is a signed complete graph of order $n$ with $k$ negative edges, $k<n-1$ and $\Sigma$ has maximum index, then the negative edges induce the signed star $K_{1,k}$.
Akbari, Dalvandi, Heydari and Maghasedi \cite{S. Akbari} proved the conjecture holds for signed complete graphs whose negative edges form a tree. Very recently, Ghorbani and Majidi \cite{E. Ghorbani} confirmed the conjecture.
In this paper, we first consider an unbalanced signed complete graph of order $n$ with $k$ negative edges whose negative edges form a spanning tree, i.e., $k=n-1$. Let $T_{a,b}$ denote the double star obtained by adding $a$ pendent vertices to one end vertex of $P_2$ and $b$ pendent vertices to the other. Then we have the following result.

\vspace*{2mm}
\begin{thm} \label{lm2.02} Let $\Sigma$ be an unbalanced signed complete graphs with order $n\geq6$ whose negative edges
form a spanning tree $T$ and maximizes the $\lambda_1(A(\Sigma))$, then $T\cong T_{1,n-3}$.
\end{thm}

Note that $A(\Sigma)$ is a real symmetric matrix.
Then the spectral radius of $\Sigma$ is the largest
absolute value of the eigenvalues of $A(\Sigma)$, denoted by $\rho(A(\Sigma))$, i.e., $\rho(A(\Sigma))=\mbox{max}\{\lambda_1(A(\Sigma)),-\lambda_n(A(\Sigma))\}$.
A further question is asked as follows.

\vspace*{2mm}
\begin{problem}\label{p111}
Which connected signed graphs with maximum $\rho(A(\Sigma))$ among the unbalanced signed complete graphs with order $n$ whose negative edges
form a spanning tree?
\end{problem}

In order to give the answer to problem \ref{p111}, a key problem is to characterize the unbalanced signed complete graph with order $n$ whose negative edges
form a spanning tree minimizes $\lambda_n(A(\Sigma))$. Up to now, a lot of researchers pay attention to $\lambda_n(A(\Sigma))$. Vijayakumar \cite{G. R. Vijayakumar} showed that any connected signed graph with smallest eigenvalue less than $-2$ has an induced signed subgraph with at most 10 vertices and smallest eigenvalue less than $-2$. Chawathe and Vijayakumar \cite{P. D. Chawathe} determined all minimal forbidden signed graphs for the class of signed graphs whose smallest eigenvalue is at least $-2$. Vijayakumar \cite{G. R. Vijayakumar1} showed that any signed graph with least eigenvalue $<-2$ contains an induced subgraph the least eigenvalue of which equals $-2$, Singhi and Vijayakumar \cite{N.M. Singhi} gave a simple proof subsequently.
Let $\mathcal{T}_{a-1,b-1}$ (see Figure 1) denote the graph obtained from $K_{1,a-1}$ and $K_{1,b-1}$ by adding an edge between two pendent vertices of them. In the paper, we determine the unbalanced signed complete graphs with order $n$ whose negative edges form a spanning tree $T$, which minimizes $\lambda_n(A(\Sigma))$.

\vspace*{2mm}
\begin{thm}\label{thm5.3}
Let $\Sigma$ be an unbalanced signed complete graphs with order $n\geq6$ whose negative edges
form a spanning tree $T$ and minimizes the $\lambda_n(A(\Sigma))$. Then $T\cong\mathcal{T}_{\lceil\frac{n}{2}\rceil-1,\lfloor\frac{n}{2}\rfloor-1}$.
\end{thm}

Let $(K_n,H^-)$ be a signed complete graph whose negative edges induce a subgraph $H$. Note that $\lambda_1(A((K_5, T^-_{1,2})))=-\lambda_n(A((K_5,P^-_5)))=3$ for $n=5$ and $-\lambda_n(A((K_n, \mathcal{T}^-_{\lceil\frac{n}{2}\rceil-1,\lfloor\frac{n}{2}\rfloor-1})))<n-2<\lambda_1(A((K_n, T^-_{1,n-3})))$ for $n\geq6$. Then by Theorems \ref{lm2.02} and \ref{thm5.3}, we finally give an answer to problem \ref{p111}.

\vspace*{2mm}
\begin{thm}\label{thm666}
Let $\Sigma$ be an unbalanced signed complete graph with order $n$ whose negative edges
form a spanning tree $T$ and maximizes the $\rho(A(\Sigma))$. Then
\begin{enumerate}
   \item[(I)] $T\cong T_{1,2}$ or $T\cong P_5$ when $n=5$;
   \item[(II)] $T\cong T_{1,n-3}$ for $n\geq6$.
 \end{enumerate}
\end{thm}

\begin{figure}
\centering
\begin{tikzpicture}[x=1.00mm, y=1.00mm, inner xsep=0pt, inner ysep=0pt, outer xsep=0pt, outer ysep=0pt,,scale=0.65]
\path[line width=0mm] (26.35,69.67) rectangle +(171.59,51.44);
\definecolor{L}{rgb}{0,0,0}
\definecolor{F}{rgb}{0,0,0}
\path[line width=0.30mm, draw=L, fill=F] (29.35,99.24) circle (0.70mm);
\path[line width=0.30mm, draw=L, fill=F] (44.35,99.24) circle (0.70mm);
\path[line width=0.30mm, draw=L, fill=F] (59.60,99.24) circle (0.70mm);
\path[line width=0.30mm, draw=L, fill=F] (70.52,110.36) circle (0.70mm);
\path[line width=0.30mm, draw=L, fill=F] (70.78,87.30) circle (0.70mm);
\path[line width=0.30mm, draw=L] (29.35,99.24) -- (45.11,99.24);
\path[line width=0.30mm, draw=L] (44.09,99.24) -- (60.87,99.24);
\path[line width=0.30mm, draw=L] (59.34,98.99) -- (70.78,87.55);
\path[line width=0.30mm, draw=L] (59.85,98.99) -- (69.76,109.66);
\path[line width=0.30mm, draw=L, fill=F] (70.27,103.05) circle (0.40mm);
\path[line width=0.30mm, draw=L, fill=F] (70.27,98.48) circle (0.40mm);
\path[line width=0.30mm, draw=L, fill=F] (70.27,92.82) circle (0.40mm);
\draw(42.82,69.75) node[anchor=base west]{\fontsize{11.38}{13.66}\selectfont $T_{1,n-3}$};
\path[line width=0.30mm, draw=L, fill=F] (90.09,110.36) circle (0.70mm);
\path[line width=0.30mm, draw=L, fill=F] (90.60,86.98) circle (0.70mm);
\path[line width=0.30mm, draw=L, fill=F] (100.26,99.18) circle (0.70mm);
\path[line width=0.30mm, draw=L, fill=F] (110.17,99.18) circle (0.70mm);
\path[line width=0.30mm, draw=L, fill=F] (120.08,99.18) circle (0.70mm);
\path[line width=0.30mm, draw=L, fill=F] (129.74,99.18) circle (0.70mm);
\path[line width=0.30mm, draw=L, fill=F] (139.90,110.11) circle (0.70mm);
\path[line width=0.30mm, draw=L, fill=F] (140.16,86.98) circle (0.70mm);
\path[line width=0.30mm, draw=L] (90.09,110.36) -- (100.77,99.18);
\path[line width=0.30mm, draw=L] (90.60,86.73) -- (100.51,99.18);
\path[line width=0.30mm, draw=L] (100.26,99.18) -- (110.17,99.18);
\path[line width=0.30mm, draw=L] (109.92,99.18) -- (130.25,99.18);
\path[line width=0.30mm, draw=L] (129.74,99.18) -- (139.65,110.11);
\path[line width=0.30mm, draw=L] (130.25,99.18) -- (140.16,86.98);
\path[line width=0.30mm, draw=L, fill=F] (90.09,102.99) circle (0.40mm);
\path[line width=0.30mm, draw=L, fill=F] (90.09,98.42) circle (0.40mm);
\path[line width=0.30mm, draw=L, fill=F] (90.09,92.82) circle (0.40mm);
\path[line width=0.30mm, draw=L, fill=F] (139.90,102.99) circle (0.40mm);
\path[line width=0.30mm, draw=L, fill=F] (139.90,98.42) circle (0.40mm);
\path[line width=0.30mm, draw=L, fill=F] (139.90,92.82) circle (0.40mm);
\path[line width=0.30mm, draw=L, fill=F] (159.98,86.98) circle (0.70mm);
\path[line width=0.30mm, draw=L, fill=F] (170.15,86.98) circle (0.70mm);
\path[line width=0.30mm, draw=L, fill=F] (180.06,86.98) circle (0.70mm);
\path[line width=0.30mm, draw=L, fill=F] (189.97,86.98) circle (0.70mm);
\path[line width=0.30mm, draw=L, fill=F] (160.24,110.11) circle (0.70mm);
\path[line width=0.30mm, draw=L, fill=F] (189.72,110.36) circle (0.70mm);
\path[line width=0.30mm, draw=L] (160.24,86.98) -- (190.22,86.98);
\path[line width=0.30mm, draw=L] (160.49,110.11) -- (170.15,86.98);
\path[line width=0.30mm, draw=L] (160.49,109.85) -- (180.06,86.98);
\path[line width=0.30mm, draw=L] (189.97,110.36) -- (170.15,86.98);
\path[line width=0.30mm, draw=L] (189.97,110.11) -- (179.80,86.98);
\path[line width=0.30mm, draw=L, fill=F] (169.13,110.11) circle (0.40mm);
\path[line width=0.30mm, draw=L, fill=F] (175.74,110.11) circle (0.40mm);
\path[line width=0.30mm, draw=L, fill=F] (182.35,110.11) circle (0.40mm);
\draw(109.15,69.75) node[anchor=base west]{\fontsize{11.38}{13.66}\selectfont $\mathcal{T}_{a-1,b-1}$};
\draw(168.37,69.75) node[anchor=base west]{\fontsize{11.38}{13.66}\selectfont $S^+_{2,n-2}$};
\draw(157.44,80.88) node[anchor=base west]{\fontsize{11.38}{13.66}\selectfont $v_1$};
\draw(168.11,80.88) node[anchor=base west]{\fontsize{11.38}{13.66}\selectfont $v_2$};
\draw(178.79,80.88) node[anchor=base west]{\fontsize{11.38}{13.66}\selectfont $v_3$};
\draw(188.44,80.88) node[anchor=base west]{\fontsize{11.38}{13.66}\selectfont $v_4$};
\draw(155.66,115.19) node[anchor=base west]{\fontsize{11.38}{13.66}\selectfont $v_5$};
\draw(186.41,115.19) node[anchor=base west]{\fontsize{11.38}{13.66}\selectfont $v_n$};
\end{tikzpicture}%
\caption{The graphs $T_{1,n-3}$, $\mathcal{T}_{a-1,b-1}$ and $S^+_{2,n-2}$.}\label{Fig. 1}
\end{figure}
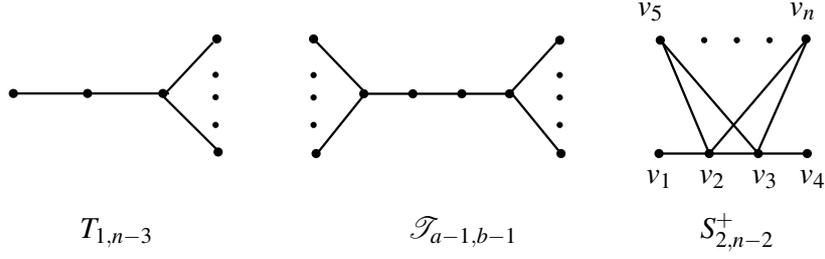

Another interesting graph matrix on signed graph is distance matrix.
The definition is given by Hameed, Shijin, Soorya, Germina and Zaslavsky \cite{S.K. Hameed}.
The sign of a path $P$ in $\Sigma$ is defined as $\sigma(P)=\Pi_{e\in E(P)}\sigma(e)$. Let $P_{(u,v)}$ denote a shortest path between vertices $u$ and $v$ and $\mathcal{P}_{(u,v)}=\{P_{(u,v)}\mid u, v~\mbox{are two given vertices of}~G\}$. Let $\sigma_{max}(u,v)=\max\{\sigma(P_{(u,v)})\mid P_{(u,v)}\in \mathcal{P}_{(u,v)}\}$ and $\sigma_{min}(u,v)=\min\{\sigma(P_{(u,v)})\mid P_{(u,v)}\in \mathcal{P}_{(u,v)}\}$. The distance between vertices $u$ and $v$, i.e., the length of a shortest path between $u$ and $v$, is denoted by $d_{uv}$. Let $d_{max}(u,v)=\sigma_{max}(u,v)d_{uv}$ and $d_{min}(u,v)=\sigma_{min}(u,v)d_{uv}$.  Two types of distance matrices in a signed graph are defined as $D^{max}(\Sigma)=(d_{max}(u,v))_{n\times n}$ and $D^{min}(\Sigma)=(d_{min}(u,v))_{n\times n}$. Two vertices $u$ and $v$ in a signed graph $\Sigma$ are said to be distance-compatible
(briefly, compatible) if $d_{max}(u,v)=d_{min}(u,v)$. And $\Sigma$ is said to be (distance-)compatible if every two vertices are compatible. Then
$D^{max}(\Sigma)=D^{min}(\Sigma)=D^{\pm}(\Sigma)$.

If we consider about signed complete graphs, then $A(\Sigma)=D^{\pm}(\Sigma)$.
The diameter of $\Sigma=(G,\sigma)$ is the diameter of its
underlying graph $G$, namely, the maximum distance between any two vertices in $G$, denoted by $d$ or $d(G)$. Our second motivation of the paper is from the following problem: what are the upper bound on eigenvalues of a signed distance matrix with diameter $d\geq2$?
In this paper, we give an upper bound on the least eigenvalues of the signed distance matrices.
Let $K_{2,n-4}=2K_1\vee(n-4)K_1$ and $2K_1=\{v_2,v_3\}$ and let $S^+_{2,n-2}$ (see Figure 1) denote the graph obtained from $K_{2,n-4}$ by adding an edge between $v_2$ and $v_3$ and attaching two pendent vertices to each vertex of $v_2$ and $v_3$, respectively.
Now we have the following result.

\vspace*{2mm}
\begin{thm}\label{lem1.11}
Let $\Sigma=(G,\sigma)$ be a connected signed graph with diameter $d(G)=d$.
Then we have the following statements.
 \begin{enumerate}
   \item[(I)] If $d=2$, then $\max\{\lambda_n(D^{max}),~\lambda_n(D^{min})\}\leq-2$ with equality holding if and only if $\Sigma$ is a balanced signed complete $k$-partite graph with $2\leq k\leq n-1$;
   \item[(II)] If $d=3$, then $\max\{\lambda_n(D^{max}),~\lambda_n(D^{min})\}\leq-2-\sqrt{2}$ with equality holding if and only if $\Sigma$ is a signed graph with underlying graph $S^+_{2,n-2}$ such that all triangles with the same sign;
   \item[(III)] If $d\geq4$, then $\max\{\lambda_n(D^{max}),\lambda_n(D^{min})\}\leq-\frac{4}{3}(d-1)$.
 \end{enumerate}
\end{thm}

It is worth to mention that Theorem \ref{lem1.11} implies that $\lambda_n(D)\leq -d$ with equality holding if and only if $G$ is a complete multipartite graph which was conjectured by Aouchiche and Hansen \cite{M. Aouchiche} and confirmed by Lin \cite{H.Q. Lin-1}.

\section{Proof of Theorem \ref{lm2.02}}\label{sec2.0}
\begin{lemma}\cite{S. Akbari}\label{lm5.3}
Let $\Sigma=(K_n, T^-_{s,t})$  be a signed complete graph and $n\geq s+t+2$. Then
\begin{align*}
P_{D^{\pm}(\Sigma)}(\lambda)&= (\lambda+1)^{n-5}({\lambda}^{5}+ \left( 5-n \right) {\lambda}^{4}+\left( 10-4\,n
 \right) {\lambda}^{3}\\
\ &+ \left( 4ku+8\,st-6\,n+10 \right) {\lambda}^{2}+
 \left(8ku+ 16\,st-4\,n+5 \right) \lambda+4ku+8st-16stu+1-n),
\end{align*}
where $k=s+t+1$ and $u=n-s-t-2$.
\end{lemma}

\begin{remark}\label{rek5.11}
Let $n=s+t+2$ and $g_{s,t}(\lambda)={\lambda}^{5}+ \left( 5-n \right) {\lambda}^{4}+\left( 10-4\,n
 \right) {\lambda}^{3}+ \left( 8\,st-6\,n+10 \right) {\lambda}^{2}+
 \left(16\,st-4\,n+5 \right) \lambda+8st+1-n$, where $1\leq s \leq t$ and $1\leq s\leq \lfloor \frac{n-2}{2}\rfloor$. Note that $g_{s-1,t+1}(\lambda)-g_{s,t}(\lambda)=-8(\lambda+1)^2(n-2s-1)<0$ for $\lambda\neq-1$, thus $\lambda_1(A((K_n, T^-_{s,t})))<\lambda_1(A((K_n, T^-_{s-1,t+1})))<\lambda_1(A((K_n, T^-_{1,n-3})))$ for $s\geq2$.

By Lemma \ref{lm5.3}, we can get the characteristic polynomial of $\Sigma=(K_n, T^-_{1,n-3})$ is
$$P_{D^{\pm}(\Sigma)}(\lambda)= (\lambda+1)^{n-3}({\lambda}^{3}+ \left( 3-n \right) {\lambda}^{2}+\left( 3-2\,n
 \right) {\lambda}+7n-23).$$
Let $g(\lambda)={\lambda}^{3}+ \left( 3-n \right) {\lambda}^{2}+\left( 3-2\,n
 \right) {\lambda}+7n-23$. Note that $g(n-2)=-(n-5)^2<0$, thus $\lambda_1(A((K_n, T^-_{1,n-3})))>n-2$ for $n\geq6$.
\end{remark}

\begin{lemma}\cite{T. Koledin} \label{lm2.01} Let $r,s,t$ and $u$ be distinct vertices of a signed graph $\Sigma$, $X=(x_1,x_2,\ldots,x_n)^T$ be an eigenvector corresponding to $\lambda_1(A(\Sigma))$. Then
\begin{enumerate}
  \item [(i)] let $\Sigma'$ be obtained from $\Sigma$ by reversing the sign
of the positive edge $rs$ and the negative edge $rt$. If
$$\left\{
    \begin{array}{ll}
       x_r\geq0,x_t\geq x_s~~or\\
      x_r\leq0,x_t\leq x_s,
    \end{array}
  \right.$$
then $\lambda_1(A(\Sigma))\leq\lambda_1(A(\Sigma'))$. If at least one inequality for the entries of $X$ is strict, then $\lambda_1(A(\Sigma))<\lambda_1(A(\Sigma'))$;
  \item [(ii)] let $\Sigma'$ be obtained from $\Sigma$ by reversing the sign of
the positive edge $rs$ and the negative edge $tu$. If $x_rx_s\leq x_tx_u$, then $\lambda_1(A(\Sigma))\leq\lambda_1(A(\Sigma'))$. If at least one of the entries $x_r,x_s,x_t,x_u$ is distinct from zero, then $\lambda_1(A(\Sigma))<\lambda_1(A(\Sigma'))$.
\end{enumerate}
\end{lemma}

Let $N_G(v)$ or $N(v)$ denote the neighbor set of vertex $v$ of $G$ and $d_G(v)=\mid N_G(v)\mid$ be the degree of the vertex $v$ in $G$. For any $U\subset V(G)$, let $N_G(U)$ denote the neighbor set of vertices $v\in U$ of $G$.

\renewcommand\proofname{\bf{Proof of Theorem \ref{lm2.02}}.}
\begin{proof}
Since $\Sigma$ is an unbalanced signed complete graphs, $T\ncong K_{1,n-1}$, $d(T)>2$ and $\lambda_1(A(\Sigma))<n-1$. Let $X=(x_1,x_2,\ldots,x_n)^T$ be an unit eigenvector corresponding to $\lambda_1(A(\Sigma))$.
By arranging the vertices of $\Sigma$ appropriately, we can assume that $V(\Sigma)=\{v_1,v_2,\ldots,v_n\}$ such that $x_1\leq x_2\leq \cdots\leq x_n$.

\noindent{\bf{$\underline{\mbox{Case 1. }}$}} $X$ is a nonnegative eigenvector. Obviously, $x_n>0$ since $X\neq\mathbf{0}$.
\begin{claim}\label{claim21}
$d_T(v_1)\geq2$.
\end{claim}
Otherwise, let $v_tv_1\in E(T)$, then $d_T(v_t)\geq2$ since $d(T)\geq3$. If $x_t=x_1$, then we just need to exchange the subscripts $v_t$ and $v_1$, as desired. If $x_t>x_1$, then we can construct a new signed graph $\Sigma'$ by reversing the sign of the positive edge $v_1w$ and the negative edge $v_tw$ for any $w\in N_T(v_t)\setminus\{v_1\}$ whose negative edges also form a spanning tree such that $\lambda_1(A(\Sigma))<\lambda_1(A(\Sigma'))$ by Lemma \ref{lm2.01}, a contradiction.

Let $x_p=\mbox{min}_{w\in N_T(v_1)}x_w$ and $x_q=\mbox{max}_{w\in N_T(v_1)}x_w$.
\begin{claim}\label{claim22}
$p=2.$
\end{claim}
Otherwise, if $x_p=x_2$, then we just need to exchange the subscripts $v_p$ and $v_2$, as desired. Next, suppose that $0\leq x_2<x_p\leq x_n$ and $v_2v_1\notin E(T)$. Note that there exists an unique path $P$ between $v_p$ and $v_2$ in $T$ since $T$ is connected. If $v_1\notin V(P)$, then a new signed graph $\Sigma'$ can be obtained from $\Sigma$ by reversing the sign of the positive edge $v_1v_2$ and the negative edge $v_1v_p$ whose negative edges also form a spanning tree such that $\lambda_1(A(\Sigma))<\lambda_1(A(\Sigma'))$ by Lemma \ref{lm2.01}, a contradiction.
If $v_1\in V(P)$, it means that we can find another vertex $v_t$ such that $P=v_2\cdots v_tv_1v_p$, then $p<t$ and $0\leq x_2< x_p\leq x_t$ clearly. And so a new signed graph $\Sigma'$ can be obtained from $\Sigma$ by reversing the sign of the positive edge $v_1v_2$ and the negative edge $v_1v_t$ whose negative edges also form a spanning tree such that $\lambda_1(A(\Sigma))<\lambda_1(A(\Sigma'))$ by Lemma \ref{lm2.01}, a contradiction.

\begin{claim}\label{claim23}
$x_q>0.$
\end{claim}
It is nature when $q=n$. Next, suppose that $x_q=0$ and $q<n$. Then $x_1=\cdots=x_q=0$ and $\sum\limits_{i=q+1}^nx_i=\lambda_1(A(\Sigma))x_1=0$, it means $X=\textbf{0}$, a contradiction.

\begin{claim}\label{claim23}
$x_i>x_1\geq0$ for $i\geq2$.
\end{claim}
We only need to prove $x_2>x_1$. Otherwise, we first revers the sign of the positive edge $v_2v_q$ and the negative edge $v_1v_q$, and then exchange the subscripts $v_1$ and $v_2$, the new signed graph is denoted by $\Sigma'$ whose negative edges also form a spanning tree, but $\lambda_1(A(\Sigma))<\lambda_1(A(\Sigma'))$ by Lemma \ref{lm2.01}, a contradiction.

\begin{claim}\label{claim24}
The subscripts of the neighbourhoods of $v_1$ in $T$ must be continuous.
\end{claim}
Otherwise, let $t$ be the minimum subscript such that $v_tv_1\notin E(T)$ and $2<t<q$. If $x_t=x_q$, then we just need to exchange the subscripts $v_q$ and $v_t$, as desired. Next, let $0\leq x_t<x_q$. Since $T$ is connected, there exists an unique path $P$ between $v_q$ and $v_t$ in $T$. We first notice that $v_2\in V(P)$ and $v_1\notin V(P)$ is impossible, otherwise, there exists a cycle which contains $v_1,v_2,v_q$ in $T$, a contradiction. If $\{v_1,v_2\}\nsubseteq V(P)$, then a new signed graph $\Sigma'$ can be obtained from $\Sigma$ by reversing the sign of the positive edge $v_1v_t$ and the negative edge $v_1v_q$ whose negative edges also form a spanning tree such that $\lambda_1(A(\Sigma))<\lambda_1(A(\Sigma'))$ by Lemma \ref{lm2.01}, a contradiction.
If $\{v_1,v_2\}\subset V(P)$ and $v_2v_t\in E(T)$, then a new signed graph $\Sigma'$ can be obtained from $\Sigma$ by reversing the sign of the positive edge $v_1v_t$ and the negative edge $v_2v_t$ whose negative edges also form a spanning tree such that $\lambda_1(A(\Sigma))<\lambda_n(A(\Sigma'))$ by Lemma \ref{lm2.01} and Claim \ref{claim23}, a contradiction.
Assume that $\{v_1,v_2\}\subset V(P)$ but $v_2v_t\notin E(T)$, this means that we can find a vertex $v_s$ such that $v_2v_s\in E(T)$, so $t<s$. Then a new signed graph $\Sigma'$ can be obtained from $\Sigma$ by reversing the sign of the positive edge $v_1v_t$ and the negative edge $v_2v_s$ whose negative edges also form a spanning tree such that $\lambda_1(A(\Sigma))<\lambda_1(A(\Sigma'))$ by Lemma \ref{lm2.01} and Claim \ref{claim23}, a contradiction.
Now, there is only one case, i.e., $v_2\notin V(P)$ and $v_1\in V(P)$, it means that there is a vertex $v_s$ such that $v_sv_1\in E(T)$ and $v_s\in V(P)$. Thus, $s>2$. If $v_tv_s\in E(T)$, then a new signed graph $\Sigma'$ can be obtained from $\Sigma$ by reversing the sign of the positive edge $v_1v_t$ and the negative edge $v_sv_t$ whose negative edges also form a spanning tree such that $\lambda_1(A(\Sigma))<\lambda_1(A(\Sigma'))$ by Lemma \ref{lm2.01} and Claim \ref{claim23}, a contradiction.
If $v_tv_s\notin E(T)$, then there is a vertex $v_r$ such that $v_rv_s\in E(T)$ and $v_r\in V(P)$, and so $r>t$. Then a new signed graph $\Sigma'$ can be obtained from $\Sigma$ by reversing the sign of the positive edge $v_1v_t$ and the negative edge $v_sv_r$ whose negative edges also form a spanning tree such that $\lambda_1(A(\Sigma))<\lambda_1(A(\Sigma'))$ by Lemma \ref{lm2.01} and Claim \ref{claim23}, a contradiction.

\begin{claim}\label{claim25}
$q=n-1$.
\end{claim}
Otherwise, we consider the vertex $v_{q+1}$ and $q+1\leq n-1$ since $d(T)\geq3$. Let $P$ be the path of $T$ between $v_1$ and $v_{q+1}$, then there must exist a vertex $v_s\in N_T(v_1)$ such that $v_s\in V(P)$. If $v_sv_{q+1}\in E(T)$, then a new signed graph $\Sigma'$ can be obtained from $\Sigma$ by reversing the sign of the positive edge $v_1v_{q+1}$ and the negative edge $v_sv_{q+1}$ whose negative edges also form a spanning tree such that $\lambda_1(A(\Sigma))<\lambda_1(A(\Sigma'))$ by Lemma \ref{lm2.01} and Claim \ref{claim23}, a contradiction. If $v_sv_{q+1}\notin E(T)$, then there exists another vertex $v_r\in N_T(v_s)$ such that $v_r\in V(P)$, and $r>q+1$ obviously. Then a new signed graph $\Sigma'$ can be obtained from $\Sigma$ by reversing the sign of the positive edge $v_1v_{q+1}$ and the negative edge $v_sv_r$ whose negative edges also form a spanning tree such that $\lambda_1(A(\Sigma))<\lambda_1(A(\Sigma'))$ by Lemma \ref{lm2.01} and Claim \ref{claim23}, a contradiction.

\begin{claim}\label{claim26}
$v_2v_n\in E(T)$.
\end{claim}
Otherwise, suppose that $v_sv_n\in E(T)$ and $s>2$, then a new signed graph $\Sigma'$ can be obtained from $\Sigma$ by reversing the sign of the positive edge $v_2v_n$ and the negative edge $v_sv_n$ whose negative edges also form a spanning tree such that $\lambda_1(A(\Sigma))<\lambda_1(A(\Sigma'))$ by Lemma \ref{lm2.01}, a contradiction.

By Claims \ref{claim21}-\ref{claim26}, $T\cong T_{1,n-3}$ immediately.
Note that $-X$ must be an eigenvector of $\Sigma$ if $X$ is an eigenvector. Thus the case that $X$ is a nonpositive eigenvector is not considered independently.

\noindent{\bf{$\underline{\mbox{Case 2. }}$}} $X$ is neither nonpositive nor nonnegative. Let $V_+=\{v_i|x_{v_i}\geq0\}$ and $V_-=\{v_i|x_{v_i}<0\}$. Clearly, $V_+,V_-\neq\emptyset$ and there must exist a vertex $v$ such that  $x_v>0$.
Let $\mid V_+\mid=a$, $\mid V_-\mid=b$ and $x_1\leq\ldots\leq x_b<0\leq x_{b+1}\leq\ldots\leq x_n$. For convenience, set $u_i=v_{n-(i-1)}$ and $y_i=x_{n-(i-1)}$ for $1\leq i\leq a$, then $x_1\leq\ldots\leq x_b<0\leq y_a\leq\ldots\leq y_1$ and $y_1>0$.

\noindent{\bf{$\underline{\mbox{Subcase 2.1. }}$}} $a\geq2$ and $b\geq2$.
\begin{claim}\label{claim27}
$V_+$ and $V_-$ are independent sets in $T$.
\end{claim}
Without loss of generality, we only consider $V_+$. Assume the contrary, there must exist three vertices $v_r,$ $u_s$ and $u_t$ such that $\{v_ru_s, u_su_t\}\subset E(T)$. Then a new signed graph $\Sigma'$ can be obtained from $\Sigma$ by reversing the sign of the positive edge $v_ru_t$ and the negative edge $u_su_t$ whose negative edges also form a spanning tree such that $\lambda_1(A(\Sigma))<\lambda_1(A(\Sigma'))$ by Lemma \ref{lm2.01}, a contradiction.

\begin{claim}\label{claim27}
$u_1v_1\in E(T)$.
\end{claim}
Otherwise, let $u_1v_s\in E(T)$ but $s>1$. If $x_1=x_s$, then we just need to exchange the subscripts $v_1$ and $v_s$, as desired. If $x_1<x_s<0$, then a new signed graph $\Sigma'$ can be obtained from $\Sigma$ by reversing the sign of the positive edge $v_1u_1$ and the negative edge $v_su_1$ whose negative edges also form a spanning tree such that $\lambda_1(A(\Sigma))<\lambda_1(A(\Sigma'))$ by Lemma \ref{lm2.01}, a contradiction.

\begin{claim}\label{claim28}
$u_1v_i\in E(T)$ for all $1\leq i\leq b$.
\end{claim}
Assume the contrary, let $u_1v_t\notin E(T)$, $t>1$ clearly. Then there must be a vertex $u_s$ and a path $P$ between $u_1$ and $v_t$ in $T$ such that $u_sv_t\in E(P)$. We
construct a new signed graph $\Sigma'$ from $\Sigma$ by reversing the sign of the positive edge $u_1v_t$ and the negative edge $u_sv_t$ whose negative edges also form a spanning tree such that $\lambda_1(A(\Sigma))<\lambda_1(A(\Sigma'))$ by Lemma \ref{lm2.01}, a contradiction.

Let $W=N_T(V_+\setminus \{u_1\})=\{v_1,\ldots,v_s\}$ for $1\leq s\leq b$. By Claim \ref{claim28}, we have every vertex in $V_+\setminus \{u_1\}$ is adjacent to just one vertex in $W$ and $x_1=\cdots=x_s$ by Lemma \ref{lm2.01}. Without loss of generality, let $d_T(v_1)\geq\cdots\geq d_T(v_s)$. For any $1\leq i<j\leq s$, let $\{u_kv_i,u_lv_j\}\subset E(T)$. Set $\lambda_1=\lambda_1(A(\Sigma))$, by characteristic equation we get
$$\left\{
    \begin{array}{ll}
      (\lambda_1+1)(x_i-x_j)=2(\sum\limits_{w\in N_T(v_j)\setminus \{u_1\}}y_w-\sum\limits_{w\in N_T(v_i)\setminus \{u_1\}}y_w),\\
      (\lambda_1+1)(y_k-y_l)=2(x_j-x_i).
    \end{array}
  \right.
$$
Obviously, $\sum\limits_{w\in N_T(v_j)\setminus \{u_1\}}y_w=\sum\limits_{w\in N_T(v_i)\setminus \{u_1\}}y_w$ and $y_k=y_l$ since $x_i=x_j$. Thus, $y_2=\cdots=y_a$ and $d_T(v_1)=\cdots=d_T(v_s).$

If $s=1$, then $T\cong T_{a-1,b-1}$, but $\lambda_1(A((K_n, T^-_{a-1,b-1})))\leq \lambda_1(A((K_n, T^-_{1,n-3})))$ with equality if and only if one of $a$ and $b$ equals to 2 by Remark \ref{rek5.11}, and we are done. We assume $s\geq2$ following. If $y_2=\cdots=y_a>0$, then we can construct a new signed graph $\Sigma'$ from $\Sigma$ by reversing the sign of the positive edge $v_1w$ and the negative edge $v_iw$ whose negative edges also form a spanning tree, where $2\leq i\leq s$ and $w\in N_T(v_i)$, such that $\lambda_1(A(\Sigma))<\lambda_1(A(\Sigma'))$ by Lemma \ref{lm2.01}, a contradiction. If $y_2=0$. Note that $y_1>0.$ By characteristic equation, we have
$$\left\{
    \begin{array}{ll}
      \lambda_1y_1=-\sum\limits_{i=1}^bx_i,\\
      y_1=x_j-\sum\limits_{i=1,i\neq j}^bx_i,\\
      \lambda_1x_j=-y_1+\sum\limits_{i=1,i\neq j}^bx_i,
    \end{array}
  \right.
$$
by the above equations, we have$$\lambda_1^2-2\lambda_1-3=0,$$thus $\lambda_1=3$. However, $3<\lambda_1(A((K_n,T^-_{1,n-3})))$ for $n\geq6$ by Remark \ref{rek5.11}, a contradiction. Thus, $T\cong T_{1,n-3}$.

\noindent{\bf{$\underline{\mbox{Subcase 2.2. }}$}} $a=1$ or $b=1$.
If $a=1$, then $u_1v_1\in E(T)$ by using a similar discussion as Claim \ref{claim27} firstly. Since $T\ncong K_{1,n-1}$, there exists at least one vertex which is not adjacent to $u_1$. Let $U=V^-\setminus N_T(u_1)$. If $\mid U\mid\geq2$, then there are two vertices $v\in N_T(u_1)$ and $w\in U$ such that $vw\in E(T)$. And we can construct a new signed graph $\Sigma'$ from $\Sigma$ by reversing the sign of the positive edge $u_1w$ and the negative edge $vw$ whose negative edges also form a spanning tree, but $\lambda_1(A(\Sigma))<\lambda_1(A(\Sigma'))$ by Lemma \ref{lm2.01}, a contradiction. Thus, $\mid U\mid=1$. By similar discussion with Claim \ref{claim26}, we have $v_2v_n\in E(T)$, that is, $T\cong T_{1,n-3}$. $b=1$ is similar. This completes the proof.
\end{proof}

\section{Proof of Theorem \ref{thm5.3}}\label{sec5}
We first give the following result about $\lambda_n(A(\Sigma))$, which are the most used tools in the identifications of graphs with minimum least eigenvalue.
\begin{lemma} \label{lm5.11} Let $r,s,t$ and $u$ be distinct vertices of a signed graph $\Sigma$, $X=(x_1,x_2,\ldots,x_n)^T$ be an unit eigenvector corresponding to $\lambda_n(A(\Sigma))$. Then
\begin{enumerate}
  \item [(i)] let $\Sigma'$ be obtained from $\Sigma$ by reversing the sign
of the positive edge $rs$ and the negative edge $rt$. If $x_r(x_s-x_t)\geq 0$, then $\lambda_n(A(\Sigma))\geq\lambda_n(A(\Sigma'))$. If $x_r\neq0$ or $x_s\neq x_t$, then $\lambda_n(A(\Sigma))>\lambda_n(A(\Sigma'))$;
  \item [(ii)] let $\Sigma'$ be obtained from $\Sigma$ by reversing the sign of
the positive edge $rs$ and the negative edge $tu$. If $x_rx_s\geq x_tx_u$, then $\lambda_n(A(\Sigma))\geq\lambda_n(A(\Sigma'))$. If at least one of the entries $x_r,x_s,x_t,x_u$ is distinct from zero, then $\lambda_n(A(\Sigma))>\lambda_n(A(\Sigma'))$.
\end{enumerate}
\end{lemma}
\renewcommand\proofname{\bf{ Proof}}
\begin{proof} (i) Firstly, we have
\begin{align*}
\lambda_n(A(\Sigma))-\lambda_n(A(\Sigma'))&\geq X^T(A(\Sigma)-A(\Sigma'))X\\
\ &=4x_r(x_s-x_t)\\
\ &\geq0.
\end{align*}
If $\lambda_n(A(\Sigma))=\lambda_n(A(\Sigma'))$, then $X$ is also an eigenvector of $A(\Sigma')$ corresponding to $\lambda_n(A(\Sigma'))$. If $x_r\neq0$ (resp. $x_s\neq x_t$), then the eigenvalue equation cannot hold
for $s$ (resp. $r$) in both signed graphs, and we are done.

(ii) \begin{align*}
\lambda_n(A(\Sigma))-\lambda_n(A(\Sigma'))&\geq X^T(A(\Sigma)-A(\Sigma'))X\\
\ &=4(x_rx_s-x_tx_u)\\
\ &\geq0.
\end{align*}
If $\lambda_n(A(\Sigma))=\lambda_n(A(\Sigma'))$, then $X$ is also an eigenvector of $A(\Sigma')$ corresponding to $\lambda_n(A(\Sigma'))$. Without loss of generality, if $x_r\neq0$, then the eigenvalue equation cannot hold
for $s$ in both signed graphs, and we are done.
\end{proof}
$G[S]$ denotes the subgraph of $G$ induced by $S$, where $S\subset V(G)$.
\begin{lemma} \label{lm5.1} Let $\Sigma$ be an unbalanced signed complete graphs with order $n$ whose negative edges
form a spanning tree $T$ and minimizes the $\lambda_n(A(\Sigma))$, then the negative edges form $T_{1,n-3}$ or $\mathcal{T}_{a-1,b-1}$, where $a+b=n$.
\end{lemma}
\begin{proof}
Let $X=(x_1,x_2,\ldots,x_n)^T$ be an unit eigenvector corresponding to $\lambda_n(A(\Sigma))$. Since $\Sigma$ is an unbalanced signed complete graphs, $T\ncong K_{1,n-1}$ and $d(T)>2$. And $\lambda_n(A(\Sigma))\leq\lambda_n(A((K_4,P^-_4)))<-2$ by Lemma \ref{lm1}.
By arranging the vertices of $\Sigma$ appropriately, we can assume that $V(\Sigma)=\{v_1,v_2,\ldots,v_n\}$ such that $x_1\leq x_2\leq \cdots\leq x_n$.

\noindent{\bf{$\underline{\mbox{Case 1. }}$}} $X$ is a nonnegative eigenvector. Obviously, $x_n>0$ since $X\neq\mathbf{0}$.
\begin{claim}\label{claim5.1}
$d_T(v_n)\geq2$.
\end{claim}
Otherwise, suppose that $v_tv_n\in E(T)$, then $d_T(v_t)\geq2$ since $d(T)\geq3$. If $x_t=x_n$, then we just need to exchange the subscripts $v_t$ and $v_n$, as desired. If $x_t<x_n$, then we can construct a new signed graph $\Sigma'$ by reversing the sign of the positive edge $v_nw$ and the negative edge $v_tw$ for any $w\in N_T(v_t)\setminus\{v_n\}$ whose negative edges also form a spanning tree such that $\lambda_n(A(\Sigma))>\lambda_n(A(\Sigma'))$ by Lemma \ref{lm5.11}, a contradiction.

Let $x_p=\mbox{min}_{w\in N_T(v_n)}x_w$, $x_q=\mbox{max}_{w\in N_T(v_n)}x_w$.
\begin{claim}\label{claim5.2}
$q=n-1.$
\end{claim}
Otherwise, if $x_q=x_{n-1}$, then we just need to exchange the subscripts $v_q$ and $v_{n-1}$, as desired. Next, let $0\leq x_q<x_{n-1}\leq x_n$ and $v_{n-1}v_n\notin E(T)$. Note that there exists an unique path $P$ between $v_q$ and $v_{n-1}$ in $T$. If $v_n\notin V(P)$, then a new signed graph $\Sigma'$ can be obtained from $\Sigma$ by reversing the sign of the positive edge $v_{n-1}v_n$ and the negative edge $v_qv_n$ whose negative edges also form a spanning tree such that $\lambda_n(A(\Sigma))>\lambda_n(A(\Sigma'))$ by Lemma \ref{lm5.11}, a contradiction.
If $v_n\in V(P)$, it means that we can find another vertex $v_t$ such that $P=v_{n-1}\cdots v_tv_nv_q$, thus $t<q<n-1$ and $0\leq x_t\leq x_q<x_{n-1}\leq x_n$ clearly. Then a new signed graph $\Sigma'$ can be obtained from $\Sigma$ by reversing the sign of the positive edge $v_{n-1}v_n$ and the negative edge $v_tv_n$ whose negative edges also form a spanning tree such that $\lambda_n(A(\Sigma))>\lambda_n(A(\Sigma'))$ by Lemma \ref{lm5.11}, a contradiction.
\begin{claim}\label{claim5.3}
The subscripts of the neighbourhoods of $v_n$ in $T$ must be continuous.
\end{claim}
Otherwise, suppose that $t$ is the maximum subscript such that $v_tv_n\notin E(T)$ and $p<t<n-1$. If $x_t=x_p$, then we just need to exchange the subscripts $v_p$ and $v_t$, as desired. Next, let $0\leq x_p<x_t\leq x_{n-1}\leq x_n$. Since $T$ is connected, there exists an unique path $P$ between $v_p$ and $v_t$ in $T$. We first notice that $v_{n-1}\in V(P)$ and $v_n\notin V(P)$ is impossible, otherwise, there exists a cycle which contains $v_p,v_{n-1},v_n$ in $T$, a contradiction. If $\{v_{n-1},v_n\}\nsubseteq V(P)$, then a new signed graph $\Sigma'$ can be obtained from $\Sigma$ by reversing the sign of the positive edge $v_tv_n$ and the negative edge $v_pv_n$ whose negative edges also form a spanning tree such that $\lambda_n(A(\Sigma))>\lambda_n(A(\Sigma'))$ by Lemma \ref{lm5.11}, a contradiction.
If $\{v_{n-1},v_n\}\subset V(P)$ and $v_tv_{n-1}\in E(T)$, then a new signed graph $\Sigma'$ can be obtained from $\Sigma$ by reversing the sign of the positive edge $v_tv_n$ and the negative edge $v_tv_{n-1}$ whose negative edges also form a spanning tree such that $\lambda_n(A(\Sigma))>\lambda_n(A(\Sigma'))$ by Lemma \ref{lm5.11}, a contradiction.
 Suppose that $\{v_{n-1},v_n\}\subset V(P)$ but $v_tv_{n-1}\notin E(T)$, this means we can find a vertex $v_s$ such that $v_sv_{n-1}\in E(T)$, so $s<p$. Then a new signed graph $\Sigma'$ can be obtained from $\Sigma$ by reversing the sign of the positive edge $v_tv_n$ and the negative edge $v_sv_{n-1}$ whose negative edges also form a spanning tree such that $\lambda_n(A(\Sigma))>\lambda_n(A(\Sigma'))$ by Lemma \ref{lm5.11}, a contradiction.

Now, there is only one case, i.e., $v_{n-1}\notin V(P)$ and $v_n\in V(P)$, it means that there is a vertex $v_s$ such that $v_sv_n\in E(T)$ and $v_s\in V(P)$. If $v_tv_s\in E(T)$, then a new signed graph $\Sigma'$ can be obtained from $\Sigma$ by reversing the sign of the positive edge $v_tv_n$ and the negative edge $v_sv_n$ whose negative edges also form a spanning tree such that $\lambda_n(A(\Sigma))>\lambda_n(A(\Sigma'))$ by Lemma \ref{lm5.11}, a contradiction.
If $v_tv_s\notin E(T)$, then
there is a vertex $v_r$ such that $v_rv_s\in E(T)$ and $v_r\in V(P)$, and so $r<p$. Then a new signed graph $\Sigma'$ can be obtained from $\Sigma$ by reversing the sign of the positive edge $v_tv_n$ and the negative edge $v_rv_n$ whose negative edges also form a spanning tree such that $\lambda_n(A(\Sigma))>\lambda_n(A(\Sigma'))$ by Lemma \ref{lm5.11}, a contradiction.
\begin{claim}\label{claim5.4}
$x_1>0$ and $p=2$.
\end{claim}
Since $d(T)>2$, $d_T(v_n)\leq n-2$. Let $x_1=0$. If there exists a vertex $v_t$ such that $x_t>0$ and $t<p$, then we can find a vertex $v_s\in N_T(v_n)$ such that $v_s$ belongs to the unique path $P$ between $v_t$ and $v_n$. If $v_tv_s\in E(T)$, then a new signed graph $\Sigma'$ can be obtained from $\Sigma$ by reversing the sign of the positive edge $v_tv_n$ and the negative edge $v_tv_s$ whose negative edges also form a spanning tree such that $\lambda_n(A(\Sigma))>\lambda_n(A(\Sigma'))$ by Lemma \ref{lm5.11}, a contradiction.
If $v_tv_s\notin E(T)$, then there exists a vertex $v_r\in V(P)$ such that $v_rv_t\in E(T)$ and a new signed graph $\Sigma'$ can be obtained from $\Sigma$ by reversing the sign of the positive edge $v_tv_n$ and the negative edge $v_tv_r$ whose negative edges also form a spanning tree such that $\lambda_n(A(\Sigma))>\lambda_n(A(\Sigma'))$ by Lemma \ref{lm5.11}, a contradiction.
Thus, $v_i\in N_T(v_n)$ if $x_i>0$.

Let $s<t<n$ and $x_s=0<x_t\leq x_n$, then
$$\lambda_n(A(\Sigma))x_s=\sum\limits_{i=p,i\neq t}^nx_i-x_t,$$
it leads to $x_t=\sum\limits_{i=p,i\neq t}^nx_i$. And $\lambda_n(A(\Sigma))x_n=-\sum\limits_{i=p}^nx_i=-2x_t$, thus $x_t>x_n$ since $\lambda_n(A(\Sigma))<-2$, it is impossible. Therefore, $x_1>0$ and $p=2$.
\begin{claim}\label{claim333}
$v_1v_{n-1}\in E(T)$.
\end{claim}
Otherwise, suppose that $v_1v_s\in E(T)$ and $s<n-1$, then a new signed graph $\Sigma'$ can be obtained from $\Sigma$ by reversing the sign of the positive edge $v_1v_{n-1}$ and the negative edge $v_1v_s$ whose negative edges also form a spanning tree such that $\lambda_n(A(\Sigma))>\lambda_n(A(\Sigma'))$ by Lemma \ref{lm5.11}, a contradiction.

Obviously, $T\cong T_{1,n-3}$ by the above claims. The proof of Case 1 also shows that $X>0$ if $X$ nonnegative. And we know that $-X$ must be an eigenvector of $\Sigma$ if $X$ is an eigenvector. Thus the case that $X$ is a nonpositive eigenvector is not considered independently.

\noindent{\bf{$\underline{\mbox{Case 2. }}$}} $X$ is neither nonpositive nor nonnegative. Let $V_+=\{v_i|x_{v_i}\geq0\}$ and $V_-=\{v_i|x_{v_i}<0\}$. Clearly, $V_+,V_-\neq\emptyset$ and there must exist a vertex $v$ such that  $x_v>0$.
Let $\mid V_+\mid=a$, $\mid V_-\mid=b$ and $x_1\leq\ldots\leq x_b<0\leq x_{b+1}\leq\ldots\leq x_n$. For convenience, set $u_i=v_{n-(i-1)}$ and $y_i=x_{n-(i-1)}$ for $1\leq i\leq a$, then $x_1\leq\ldots\leq x_b<0\leq y_a\leq\ldots\leq y_1$.
\begin{claim}\label{claim5.5}
There exists an unique one negative edge $e=u_av_b$ between $V_+$ and $V_-$.
\end{claim}
Otherwise, let $e_1=u_sv_k$ and $e_2=u_tv_l$ be two negative edges between $V_+$ and $V_-$. If $e_1\cap e_2\neq\emptyset$, without loss of generality, let $u_s=e_1\cap e_2$, then constructing a new signed graph $\Sigma'$ from $\Sigma$ by reversing the sign of the positive edge $v_kv_l$ and the negative edge $u_sv_k$ whose negative edges also form a spanning tree, and $\lambda_n(A(\Sigma))>\lambda_n(A(\Sigma'))$ by Lemma \ref{lm5.11}, a contradiction. Let $e_1\cap e_2=\emptyset$ next. Since there is only one path containing the four vertices, we can find two of them as the end vertices of the path. Without loss of generality, let $v_k,v_l$ be the vertices we found, and $v_kv_l\notin E(T)$, then constructing a new signed graph $\Sigma'$ from $\Sigma$ by reversing the sign of the positive edge $v_kv_l$ and the negative edge $u_sv_k$ whose negative edges also form a spanning tree, and $\lambda_n(A(\Sigma))>\lambda_n(A(\Sigma'))$ by Lemma \ref{lm5.11}, a contradiction. Thus, there exists only one negative edge $e=u_sv_t$ between $V_+$ and $V_-$. Suppose that $s=a$ but $t\neq b$. If $x_t=x_b$, we just need to exchange the subscripts $v_t$ and $v_b$, as desired. If $x_t<x_b$, then constructing a new signed graph $\Sigma'$ from $\Sigma$ by reversing the sign of the positive edge $u_av_b$ and the negative edge $u_av_t$ whose negative edges also form a spanning tree, and $\lambda_n(A(\Sigma))>\lambda_n(A(\Sigma'))$ by Lemma \ref{lm5.11}, a contradiction. We can get the similar contradictions when $s\neq a$ but $t=b$ or $s\neq a$ and $t\neq b$ by Lemma \ref{lm5.11}. And we are done.

By the similar discussion as Case 1 on $V_+$ and $V_-$, the following claim holds immediately.
\begin{claim}\label{claim5.6}
$T[V_+]=K_{1,a-1}$ with $d_T(u_1)=a-1$ and $T[V_1]=K_{1,b-1}$ with $d_T(v_b)=b-1$.
\end{claim}

By Claims \ref{claim5.5} and \ref{claim5.6}, $T\cong T_{1,n-3}$ if $a=1$ or $b=1$ and $T\cong \mathcal{T}_{a-1,b-1}$ if $a,b\geq2$. This completes the proof.
\end{proof}

\begin{lemma} \label{lm5.2} Let $\Sigma=(K_n,\mathcal{T}^-_{a-1,b-1})$ with order $n=a+b\geq6$, where $2\leq b\leq \lfloor \frac{n}{2}\rfloor$. Then
\begin{align*}
P_{D^{\pm}(\Sigma)}(\lambda)&= (\lambda+1)^{n-6}(\lambda-1)({\lambda}^{5}+ \left( 9-n \right) {\lambda}^{4}+\left( 34-6\,n
 \right) {\lambda}^{3}\\
\ &+ \left( 8\,ab-24\,n+58 \right) {\lambda}^{2}+
 \left( 16\,ab-26\,n-3 \right) \lambda-24\,ab+89\,n-323).
\end{align*}
Furthermore, if $a=\lceil\frac{n}{2}\rceil$ and $b=\lfloor \frac{n}{2}\rfloor$, then $\lambda_n(A(\Sigma))>2-n$ for $n\geq6$ and $\lambda_n(A(\Sigma))<-4$ for $n\geq8$.
\end{lemma}
\begin{proof}
$$D^{\pm}(\Sigma)=\left(
                    \begin{array}{cccccc}
                      0 & \mathbf{1}_{a-2} & -1 & -1 & \mathbf{1}_{b-2} & 1 \\
                      \mathbf{1}^T_{a-2} & J_{a-2}-I_{a-2} & -\mathbf{1}^T_{a-2} & \mathbf{1}^T_{a-2} & J_{(a-2)\times(b-2)} & \mathbf{1}^T_{a-2} \\
                      -1 & -\mathbf{1}_{a-2} & 0 & 1 & \mathbf{1}_{b-2} & 1 \\
                      -1 & \mathbf{1}_{a-2} & 1 & 0 & \mathbf{1}_{b-2} & -1 \\
                      \mathbf{1}^T_{b-2} & J_{(b-2)\times(a-2)} &  \mathbf{1}^T_{b-2} &  \mathbf{1}^T_{b-2} & J_{b-2}-I_{b-2} & -\mathbf{1}^T_{b-2} \\
                      1 & \mathbf{1}_{a-2} & 1 & -1 & -\mathbf{1}_{b-2} & 0 \\
                    \end{array}
                  \right)
,$$
then $det(\lambda I-D^{\pm}(\Sigma))=(\lambda+1)^{n-6}(\lambda-1)f_{a,b}(\lambda)$, where
$f_{a,b}(\lambda)={\lambda}^{5}+ \left( 7-n \right) {\lambda}^{4}+ \left( 22-6\,n
 \right) {\lambda}^{3}\\+ \left( 8\,ab-16\,n+18 \right) {\lambda}^{2}+
 \left( 16\,ab-10\,n-39 \right) \lambda-24\,ab+65\,n-169.$
Set $\alpha=\lambda_n(A(\Sigma))$. Since $f_{a,b}(-3)=32n-160>0$ for $n\geq6$, $\alpha<-3$. Note that $f_{a+1,b-1}(\lambda)-f_{a,b}(\lambda)=-8\, \left( \lambda+3 \right)  \left( \lambda-1 \right)\\  \left( a-b+1 \right)$, so
$f_{a+1,b-1}(\alpha)-f_{a,b}(\alpha)<0$ for $n\geq6$. It follows that $\alpha$ decreases as $b$ increases. If $n$ is even, then $\alpha$ is minimum when $a=b=\frac{n}{2}$ and $f_{n/2,n/2}(\lambda)=h_1(\lambda)h_2(\lambda),$ where $h_1(\lambda)={\lambda}^{2}+ \left( 6-n \right) \lambda-3\,n+13$ and $h_2(\lambda)={\lambda}^{3}+{\lambda}^{2}+ \left( 3-2\,n \right) \lambda+2\,n-13.$ The least eigenvalue of $h_1(\lambda)$ is $n/2-3-1/2\,\sqrt {{n}^{2}-16}>-3$, thus $\alpha$ must be the least eigenvalue of $h_2(\lambda)$.
We also find that $h_2(-4)=10n-73>0$ for $n\geq8$. Thus, we have $\alpha<-4$ for $n\geq8$. Note that $h_2(2-n)=-(n-5)(n^2-4n+1)<0$ obviously and $h'_2(\lambda)=3\lambda^2+2\lambda-2n+3$ with two stagnation point $-\frac{1\pm\sqrt{6n-8}}{3}>2-n$, thus  $\alpha>2-n$ for $n\geq6$.
If $n$ is odd, then $a=\frac{n+1}{2}$ and $b=\frac{n-1}{2}$. Using the similar method, we can also obtained $f_{a,b}(-4)=10n^2-23n-375>0$ for $n\geq8$, $f_{a,b}(2-n)=-(n-5)(n-3)(2n^3-17n^2+40n-9)<0$ and $f_{a,b}(\lambda)$ is strictly monotone increasing when $\lambda<2-n$, thus the result is still valid.
\end{proof}

\begin{remark}\label{rek5.1}
Since $\Sigma=(K_n, T^-_{1,n-3})$ is unbalanced, $\lambda_n(A(\Sigma))$ must be an eigenvalue of $g(\lambda)$ by Remark \ref{rek5.11}. Note that $g'(\lambda)=\left( \lambda+1 \right)  \left( 3\,\lambda-2\,n+3 \right) $ and $g'(\lambda)>0$ when $\lambda<-1$, and so $g(\lambda)$ is strictly monotone increasing when $\lambda<-1$. Combining this with $g(-4)=-n-51<0$, thus $\lambda_n(A(\Sigma))>-4$. If $n=6,7$, we can get $\lambda_n(A((K_n, \mathcal{T}^-_{\lceil\frac{n}{2}\rceil-1,\lfloor \frac{n}{2}\rfloor-1})))<\lambda_n(A((K_n, T^-_{1,n-3})))$ by simple calculation. If $n\geq8$, then $\lambda_n(A((K_n, \mathcal{T}^-_{\lceil\frac{n}{2}\rceil-1,\lfloor \frac{n}{2}\rfloor-1})))<-4<\lambda_n(A((K_n, T^-_{1,n-3})))$.
\end{remark}

By Lemmas \ref{lm5.1} and \ref{lm5.2} and Remark \ref{rek5.1}, we can get the Theorem \ref{thm5.3} immediately.

\section{Proof of Theorem \ref{lem1.11}}\label{sec2}
For the sake of convenience, $\sigma_{ij}$ for short $\sigma_{max}(v_i,v_j)$ and $\sigma'_{ij}$ for short $\sigma_{min}(v_i,v_j)$.
\begin{lemma}\label{lem1.12}
 Let $\Sigma=(S^+_{2,n-2},\sigma)$ be a connected signed graph and $D^{max}$ and $D^{min}$ be its signed distance matrices. If all triangles have the same sign, then $D^{max}(\Sigma)=D^{min}(\Sigma)=D^{\pm}(\Sigma)$ and
 \begin{enumerate}
   \item[(1)] $P_{D^{\pm}(\Sigma)}(\lambda)=(\lambda+2)^{n-5}\left( {\lambda}^{2}+4\,\lambda+2 \right)({\lambda}^{3}+(-2n+6)\,{\lambda}^{2}-(2n+6)\,\lambda+2n-20)$ when \\$\sigma_{23}\sigma_{35}\sigma_{25}=1$ and $\lambda_n(D^{\pm}((S^+_{2,n-2},\sigma)))=-2-\sqrt{2}$;
   \item[(2)] $P_{D^{\pm}(\Sigma)}(\lambda)=(\lambda+2)^{n-5}\left( {\lambda}^{2}-4\,\lambda-6 \right))({\lambda}^{3}+(-2n+14)\,{\lambda}^{2}-(18n-8)\,\lambda+8n-100)$ when $\sigma_{23}\sigma_{35}\sigma_{25}=-1$ and
       $\lambda_n(D^{\pm}((S^+_{2,n-2},\sigma)))=-2-\sqrt{2}$ if and only if $\Sigma=(P_4,\sigma)$.
 \end{enumerate}
\end{lemma}
\renewcommand\proofname{\bf{Proof}.}
\begin{proof}
Let $V(S^+_{2,n-2})=\{v_1,v_2,\ldots,v_n\}$. Then $\Sigma$ is compatible since all triangles have the same sign, i.e., $D^{max}(\Sigma)=D^{min}(\Sigma)=D^{\pm}(\Sigma)$ and
$$D^{\pm}((S^+_{2,n-2},\sigma))=\left(
                                  \begin{array}{ccccccc}
                                    0 & \sigma_{12} & 2\sigma_{13} & 3\sigma_{14} & 2\sigma_{15} & \cdots & 2\sigma_{1n} \\
                                    \sigma_{12} & 0 & \sigma_{23} & 2\sigma_{24} & \sigma_{25} & \cdots & \sigma_{2n} \\
                                    2\sigma_{13} & \sigma_{23} & 0 & \sigma_{34} & \sigma_{35} & \cdots & \sigma_{3n} \\
                                    3\sigma_{14} & 2\sigma_{24} & \sigma_{34} & 0 & 2\sigma_{45} & \cdots & 2\sigma_{4n} \\
                                    2\sigma_{15} & \sigma_{25} & \sigma_{35} & 2\sigma_{45} & 0 & \cdots & 2\sigma_{5n} \\
                                    \vdots & \vdots & \vdots & \vdots & \vdots & \ldots & \vdots \\
                                    2\sigma_{1n} & \sigma_{2n} & \sigma_{3n} & 2\sigma_{4n} & 2\sigma_{5n} & \cdots & 0 \\
                                  \end{array}
                                \right).$$
Note that $\sigma_{13}=\sigma_{12}\sigma_{23}$, $\sigma_{14}=\sigma_{13}\sigma_{34}$ and $\sigma_{25}\sigma_{35}\sigma_{2i}\sigma_{3i}=1$ for $6\leq i\leq n$. Then $det(\lambda I-D^{\pm}(\Sigma))=(\lambda+2)^{n-5}f(\lambda)$, where
$$f(\lambda)=\begin{vmatrix}
                        \lambda & -1 & -\sigma_{23} & -3\sigma_{23} & -2(n-4) \\
                        -1 & \lambda & -\sigma_{23} & -2\sigma_{23} & -(n-4) \\
                        -2\sigma_{23} & -\sigma_{23} & \lambda & -1 & -(n-4)\sigma_{35}\setminus \sigma_{25} \\
                        -3\sigma_{23} & -2\sigma_{23} & -1 & \lambda & -2(n-4)\sigma_{35}\setminus \sigma_{25} \\
                        -2 & -1 & -\sigma_{35}\setminus \sigma_{25} & -2\sigma_{35}\setminus \sigma_{25} & \lambda-2(n-5)
                        \end{vmatrix}.$$
If $\sigma_{23}\sigma_{35}\sigma_{25}=1$, then
$f(\lambda)=(\lambda+2-\sqrt{2})(\lambda+2+\sqrt{2})f_1(\lambda)$, where $f_1(\lambda)= {\lambda}^{3}+(-2n+6)\,{\lambda}^{2}-(2n+6)\,\lambda+2n-20$.
Note that the least stagnation point of $f_1(\lambda)$ is $\frac{2n-4-\sqrt{4n^2-18n+54}}{3}>-2-\sqrt{2}$ and $f_1(-2-\sqrt{2})<0$ for $n\geq5$. Thus $\lambda_n(D^{\pm}((S^+_{2,n-2},\sigma)))=-2-\sqrt{2}$.
If $\sigma_{23}\sigma_{35}\sigma_{25}=-1$, then
$f(\lambda)=(\lambda-2-\sqrt{10})(\lambda-2+\sqrt{10})f_2(\lambda)$, where $f_2(\lambda)= {\lambda}^{3}+(-2n+14)\,{\lambda}^{2}-(18n-8)\,\lambda+8n-100$. Note that $f_2(-2-\sqrt{2})=(10n-40)\sqrt{2}+14n-56$ equals to zero if and only if $n=4$. Then $\lambda_n(D^{\pm}((S^+_{2,n-2},\sigma)))=-2-\sqrt{2}$ if and only if $n=4$, that is, $\Sigma=(P_4,\sigma)$.
\end{proof}

\begin{lemma}\label{lm1} {\rm(Cauchy Interlace Theorem)}. Let $A$ be a Hermitian matrix with order $n$ and let $B$ be a principal submatrix of $A$ with order $m$. If $\lambda_1(A)\geq\lambda_2(A)\geq\cdots\geq\lambda_n(A)$ list the eigenvalues of $A$ and $\mu_1(B)\geq\mu_2(B)\geq\cdots\geq\mu_m(B)$ list the eigenvalues of $B$, then $\lambda_{n-m+i}(A)\leq\mu_i(B)\leq\lambda_i(A)$ for $i=1,\ldots,m$.
\end{lemma}

\renewcommand\proofname{\bf{Proof of Theorem \ref{lem1.11}}.}
\begin{proof}
(I). If $d=2$, then the result follows from (III) of Theorem \ref{thm12}. (II). Let $d=3$. Since the least signed distance eigenvalues of a signed $P_4$ equal to $-2-\sqrt{2}$, let's assume that $n\geq5$. Let $V(P_4)=\{v_1,v_2,v_3,v_4\}$, then $A_1$ is the principal submatrix of $D^{max}$ and $B_1$ is the principal submatrix of $D^{min}$ indexed by $V(P_4)$ respectively, and both $\lambda_4(A_1)$ and $\lambda_4(B_1)$ are less than or equal to $-2-\sqrt{2}$ besides the matrices as follows
$$A_1=\left(
    \begin{array}{cccc}
      0 & 1 & 2 & 3 \\
      1 & 0 & -1 & 2 \\
      2 & -1 & 0 & 1 \\
      3 & 2 & 1 & 0 \\
    \end{array}
  \right)~~~~B_1=\left(
    \begin{array}{cccc}
      0 & 1 & -2 & -3 \\
      1 & 0 & 1 & -2 \\
      -2 & 1 & 0 & 1 \\
      -3 & -2 & 1 & 0 \\
    \end{array}
  \right)
.$$
$A_1$ shows that $\sigma_{12}=1$, $\sigma_{23}=-1$ but $\sigma_{13}=1$, it means that there exists another vertex $u$ such that $\{uv_1,uv_3\}\subset E$ and $\sigma_{13}=1$. $B_1$ shows that $\sigma'_{12}=\sigma'_{23}=1$ but $\sigma'_{13}=-1$, it means that there exists another vertex $w$ such that $wv_1,wv_3\in E$ and $\sigma'_{13}=-1$. However, the least eigenvalue of the principal submatrix of $D^{max}$ (resp. $D^{min}$) indexed by $V(P_4)\cup \{u\}$ (resp. $V(P_4)\cup \{w\})$  is always less than $-2-\sqrt{2}$. Thus, both $\lambda_n(D^{max})$ and $\lambda_n(D^{min})$ are less than or equal to $-2-\sqrt{2}$ when $d=3$ by Lemma \ref{lm1}. Next, we consider the necessary and sufficient condition of the equality. The sufficiency holds obviously by Lemma \ref{lem1.12}. Conversely, we can get the following claims.

\begin{claim}\label{claim1.1}
$u\in \cup_{i=1}^{4}N(v_i)$ for any vertex $u\in V\setminus V(P_4)$.
\end{claim}
Otherwise, there exists a vertex $v_5$ such that $d_{i5}=2$ or $3$ for $1\leq i\leq4$ and $\mid\sigma_{ij}\mid=1$. Then $A_{2}$ and $B_2$ are principle submatrices of $D^{max}$ and $D^{min}$ indexed by $V(P_4)\cup \{v_5\}$, respectively. For the entries of $A_2$, we must note that $\sigma_{13}=1$ if $\sigma_{12}\sigma_{23}=1$, $\sigma_{24}=1$ if $\sigma_{23}\sigma_{34}=1$ and $\sigma_{14}=1$ if $\sigma_{12}\sigma_{23}\sigma_{34}=1$. We should also note the relationship between the elements of $B_2$, $\sigma'_{13}=-1$ if $\sigma'_{12}\sigma'_{23}=-1$, $\sigma'_{24}=-1$ if $\sigma'_{23}\sigma'_{34}=-1$ and $\sigma'_{14}=-1$ if $\sigma'_{12}\sigma'_{23}\sigma'_{34}=-1$. However, calculated by Matlab, both $\lambda_5(A_{2})$ and $\lambda_5(B_{2})$ are less than $-2-\sqrt{2}$, thus both $\lambda_n(D^{max})$ and $\lambda_n(D^{min})$ are less than $-2-\sqrt{2}$ by Lemma \ref{lm1}, a contradiction.
$$A_2=\left(
        \begin{array}{ccccc}
          0 & \sigma_{12} & 2\sigma_{13} & 3\sigma_{14} & \sigma_{15}d_{15} \\
          \sigma_{12} & 0 & \sigma_{23} & 2\sigma_{24} & \sigma_{25}d_{25} \\
          2\sigma_{13} & \sigma_{23} & 0 & \sigma_{34} & \sigma_{35}d_{35} \\
          3\sigma_{14} & 2\sigma_{24} & \sigma_{34} & 0 & \sigma_{45}d_{45} \\
          \sigma_{15}d_{15} & \sigma_{25}d_{25} & \sigma_{35}d_{35} & \sigma_{45}d_{45} & 0 \\
        \end{array}
\right)$$
$$~~B_2=\left(
        \begin{array}{ccccc}
          0 & \sigma'_{12} & 2\sigma'_{13} & 3\sigma'_{14} & \sigma'_{15}d_{15} \\
          \sigma'_{12} & 0 & \sigma'_{23} & 2\sigma'_{24} & \sigma'_{25}d_{25} \\
          2\sigma'_{13} & \sigma'_{23} & 0 & \sigma'_{34} & \sigma'_{35}d_{35} \\
          3\sigma'_{14} & 2\sigma'_{24} & \sigma'_{34} & 0 & \sigma'_{45}d_{45} \\
          \sigma'_{15}d_{15} & \sigma'_{25}d_{25} & \sigma'_{35}d_{35} & \sigma'_{45}d_{45} & 0 \\
        \end{array}
      \right)
$$

\begin{claim}\label{claim1.2}
$N(v_i)\cap N(v_2)\cap N(v_3)=\emptyset$ for $i=1,4$.
\end{claim}
By symmetry, we only need to prove that $N(v_1)\cap N(v_2)\cap N(v_3)=\emptyset$.
 If $(H_3,\sigma)\subset \Sigma$, then the least eigenvalues of the principal submatrices of $D^{max}$ and $D^{min}$ indexed by $V(H_3)$ are always less than $-2-\sqrt{2}$ besides $A_3$ and $B_3$, respectively. From $A_3$ we find that $\sigma_{12}=1$, $\sigma_{23}=-1$ but $\sigma_{13}=1$, it means that there exists another vertex $u$ such that $\{uv_1,uv_3\}\subset E$ and $\sigma_{max}(uv_1)\sigma_{max}(uv_3)=1$. $B_3$ shows that $\sigma'_{12}=\sigma'_{23}=1$ but $\sigma'_{13}=-1$, it means that there exists another vertex $w$ such that $wv_1,wv_3\in E$ and $\sigma_{min}(wv_1)\sigma_{min}(wv_3)=-1$. However, the least eigenvalue of the principal submatrix of $D^{max}$ (resp. $D^{min}$) indexed by $V(H_3)\cup \{u\}$ (resp. $V(H_3)\cup \{w\}$) is always less than $-2-\sqrt{2}$ through Matlab, thus both $\lambda_n(D^{max})$ and $\lambda_n(D^{min})$ are less than $-2-\sqrt{2}$ by Lemma \ref{lm1}, a contradiction.
$$A_3=\left(
    \begin{array}{ccccc}
      0 & 1 & 2 & 3 & 1 \\
      1 & 0 & -1 & 2 & 1 \\
      2 & -1 & 0 & 1 & -1 \\
      3 & 2 & 1 & 0 & 2 \\
      1 & 1 & -1 & 2 & 0 \\
    \end{array}
  \right)~~B_3=\left(
    \begin{array}{ccccc}
      0 & 1 & -2 & -3 & 1 \\
      1 & 0 & 1 & -2 & 1 \\
      -2 & 1 & 0 & 1 & 1 \\
      -3 & -2 & 1 & 0 & -2 \\
      1 & 1 & 1 & -2 & 0 \\
    \end{array}
  \right)$$

\begin{claim}\label{claim1.3}
$N(v_1)\cap N(v_3)=\emptyset$ and $N(v_2)\cap N(v_4)=\emptyset$
\end{claim}
We assume that $N(v_1)\cap N(v_3)\neq\emptyset$, then $(H_2,\sigma)\subset \Sigma$ and the least eigenvalues of the principal submatrices of $D^{max}$ and $D^{min}$ indexed by $V(H_2)$ are always less than $-2-\sqrt{2}$ besides $A_4$ and $B_4$, respectively. $A_4$ indicates $\sigma_{12}=1$, $\sigma_{23}=-1$ but $\sigma_{13}=1$, it means that there exists another vertex $u$ such that $\{uv_1,uv_3\}\subset E$ and $\sigma_{max}(uv_1)\sigma_{max}(uv_3)=1$. And $B_4$ suggests $\sigma'_{12}\sigma'_{23}=1$ but $\sigma'_{13}=-1$, it means that there exists another vertex $w$ such that $wv_1,wv_3\in E$ and $\sigma_{min}(wv_1)\sigma_{min}(wv_3)=-1$. Combining this with claim \ref{claim1.2}, we can observe the least eigenvalue of the principal submatrix of $D^{max}$ (resp. $D^{min}$) indexed by $V(H_2)\cup \{u\}$ (resp. $V(H_2)\cup \{w\}$) is always less than $-2-\sqrt{2}$ through Matlab, thus both $\lambda_n(D^{max})$ and $\lambda_n(D^{min})$ are less than $-2-\sqrt{2}$ by Lemma \ref{lm1}, a contradiction.
$$A_4=\left(
    \begin{array}{ccccc}
      0 & 1 & 2 & 3 & 1 \\
      1 & 0 & -1 & 2 & 2 \\
      2 & -1 & 0 & 1 & -1 \\
      3 & 2 & 1 & 0 & 2 \\
      1 & 2 & -1 & 2 & 0 \\
    \end{array}
  \right)~~B_4=\left(
    \begin{array}{ccccc}
      0 & 1 & -2 & -3 & 1 \\
      1 & 0 & 1 & -2 & 2 \\
      -2 & 1 & 0 & 1 & 1 \\
      -3 & -2 & 1 & 0 & -2 \\
      1 & 2 & 1 & -2 & 0 \\
    \end{array}
  \right)$$

Note that the least eigenvalue of the principal submatrices of $D^{max}$ (or $D^{min}$) indexed by $V(H_4)$ and $V(H_8)$ are always less than $-2-\sqrt{2}$, respectively. Thus, the following claim is obvious.
\begin{claim}\label{claim1.4}
$N(v_1)=\{v_2\}$ and $N(v_4)=\{v_3\}$.
\end{claim}

\begin{claim}\label{claim1.5}
$N(v_2)\setminus (N(v_2)\cap N(v_3))=N(v_3)\setminus (N(v_2)\cap N(v_3))=\emptyset$.
\end{claim}
By symmetry, we only need to prove that $N(v_2)\setminus (N(v_2)\cap N(v_3))=\emptyset$. If $(H_7,\sigma)\subset \Sigma$, then the least eigenvalues of the principal submatrices of $D^{max}$ and $D^{min}$ indexed by $V(H_7)$ are always less than $-2-\sqrt{2}$ besides $A_5$ and $B_5$, respectively. By $A_5$ and $B_5$, we have $v_3v_5\notin E$ and $d(v_4,v_5)=2$, then there is another vertex $u$ such that $uv_4\in E$ by Claims \ref{claim1.1}-\ref{claim1.3}, but it contradicts to Claim \ref{claim1.4}.
$$~~A_5=\left(
    \begin{array}{ccccc}
      0 & 1 & 2 & 3 & 2 \\
      1 & 0 & -1 & 2 & 1 \\
      2 & -1 & 0 & 1 & 2 \\
      3 & 2 & 1 & 0 & 2 \\
      2 & 1 & 2 & 2 & 0 \\
    \end{array}
  \right)
  ~~B_5=\left(
    \begin{array}{ccccc}
      0 & 1 & -2 & -3 & 2 \\
      1 & 0 & 1 & -2 & 1 \\
      -2 & 1 & 0 & 1 & -2 \\
      -3 & -2 & 1 & 0 & -2 \\
      2 & 1 & -2 & -2 & 0 \\
    \end{array}
  \right)
$$

By Claims \ref{claim1.1}-\ref{claim1.5}, we conclude that $u\in N(v_2)\cap N(v_3)$ for any vertex $u\in V\setminus V(P_4)$. If $\mid N(v_2)\cap N(v_3)\mid=0~\mbox{or}~1$, then the result is follows. So let $\mid N(v_2)\cap N(v_3)\mid\geq2$ and $\{u_1,u_2\}\subseteq N(v_2)\cap N(v_3)$.
\begin{claim}\label{claim1.6}
\begin{enumerate}
  \item[(1)] $u_1u_2\notin E$;
  \item[(2)] $\sigma_{max}(u_1v_2)\sigma_{max}(u_1v_3)\sigma_{max}(u_2v_2)\sigma_{max}(u_2v_3)=1$;
  \item[(3)] $\sigma_{min}(u_1v_2)\sigma_{min}(u_1v_3)\sigma_{min}(u_2v_2)\sigma_{min}(u_2v_3)=1$.
\end{enumerate}
\end{claim}
 If $u_1u_2\in E$, then $(H_5,\sigma)\subset \Sigma$ and the least eigenvalues of the principal submatrices of $D^{max}$ and $D^{min}$ indexed by $V(H_5)$ are always less than $-2-\sqrt{2}$, a contradiction. If $\sigma_{max}(u_1v_2)\sigma_{max}(u_1v_3)$=-$\sigma_{max}(u_2v_2)\sigma_{max}(u_2v_3)$, then the least eigenvalues of the principal submatrices of $D^{max}$ and $D^{min}$ indexed by $V(H_6)$ are always less than $-2-\sqrt{2}$, a contradiction. (3) is similar.

From Claim \ref{claim1.6}, we have $\Sigma$ is a signed graph with underly graph $S^+_{2,n-2}$ such that all triangles with the same sign.

(III). Suppose that $d\geq4$ and $P_{d+1}=v_1v_2\ldots v_{d+1}$ is a diameter path
of $\Sigma$ such that $\sigma_{1,d+1}=\sigma_{1,d}\sigma_{d,d+1}$. Let $X=(x_1, x_2,\ldots, x_n)^T$
be an vector with $x_i$ corresponding to $v_i$. Assume that $x_1=1$, $x_{d}=-\sigma_{1,d}$, and $x_{d+1}=-\sigma_{1,d+1}$, then
\begin{align*}
\lambda_n(D^{max})&\leq\frac{X^TD^{max}X}{X^TX}\\
\ &=\frac{2}{3}(\sigma_{1d}(d-1)x_1x_d+\sigma_{1,d+1}dx_1x_{d+1}+\sigma_{d,d+1}x_dx_{d+1})\\
\ &=-\frac{4}{3}(d-1).
\end{align*}
And this bound still holds for $D^{min}$ clearly.
\end{proof}
\begin{remark}\label{rk111}Let $\Sigma$ is a balanced graph with diameter $d\geq4$ and $P_{d+1}=v_1v_2\ldots v_{d+1}$ is a diameter path
of $\Sigma$ such that $\sigma_{1,d+1}=\sigma_{1,d}\sigma_{d,d+1}$. Let $X=(x_1, x_2, \ldots, x_n)^T$
be an vector with $x_i$ corresponding to $v_i$. Assume that $x_1=\sigma_{1,d}$, $x_2=\sigma_{2,d}$, $x_{d}=-1$, and $x_{d+1}=-\sigma_{d,d+1}$, then
\begin{align*}
\lambda_n(D^{max})&\leq\frac{X^TD^{max}X}{X^TX}\\
\ &=\frac{1}{2}(-3d+4+\sigma_{1,2}\sigma_{1,d}\sigma_{2,d}-\sigma_{2,d+1}(d-1)\sigma_{2,d}\sigma_{d,d+1}).
\end{align*}
Note that the shortest paths of every two vertices have the same sign since $\Sigma$ is balanced, thus $\sigma_{1,d}=\sigma_{1,2}\sigma_{2,d}$ and $\sigma_{2,d+1}=\sigma_{2,d}\sigma_{d,d+1}$. Then
$\lambda_n(D^{max})\leq -2d+3.$ The conclusion still holds for $D^{min}$.
\end{remark}
\section{Conclusion remark}
As usual, let $K_n$, $P_n$ and $K_{n_1,n_2,\cdots,n_k}$, where $\sum\limits_{i=1}^kn_i=n$, denote the complete graph, the path and the complete $k$-partite graph with order $n$, respectively. The girth of $\Sigma=(G,\sigma)$, short for $g$, is the girth of its underlying graph $G$, namely, the length of a shortest cycle in $G$. Let $D(G)=(d_{uv})_{n\times n}$ be the distance matrix of a connected graph $G$. Yu \cite{Yu} investigated the least distance eigenvalue of a graph, and determined all connected graphs
with $\lambda_n(D(G))\in[-2.383,0]$ among graphs with $n$ vertices. Lin, Hong, Wang and Shu \cite{H.Q. Lin} extended the above result separately. The authors showed that $\lambda_n(D(G))=-1$ if and only if $G\cong K_n$ and $\lambda_n(D(G))=-2$ if and only if
$G\cong K_{n_1,\cdots,n_k}$. And there does not exist a graph $G$ of order $n$ with $-2<\lambda_n(D(G))<-1$. Motivated by the papers, we pose the following problem:

\begin{problem}\label{p1}
Which connected signed graphs have the least eigenvalue of the signed distance matrices belonging to $[-2,-1]$?
\end{problem}

\begin{figure}
\centering
\begin{tikzpicture}[x=1.00mm, y=1.00mm, inner xsep=0pt, inner ysep=0pt, outer xsep=0pt, outer ysep=0pt,,scale=0.6]
\path[line width=0mm] (16.70,75.47) rectangle +(217.19,67.37);
\definecolor{L}{rgb}{0,0,0}
\definecolor{F}{rgb}{0,0,0}
\path[line width=0.30mm, draw=L, fill=F] (80.18,139.84) circle (1.00mm);
\path[line width=0.30mm, draw=L, fill=F] (68.74,131.71) circle (1.00mm);
\path[line width=0.30mm, draw=L, fill=F] (91.36,131.71) circle (1.00mm);
\path[line width=0.30mm, draw=L, fill=F] (73.83,120.27) circle (1.00mm);
\path[line width=0.30mm, draw=L, fill=F] (85.77,120.27) circle (1.00mm);
\path[line width=0.30mm, draw=L] (79.93,140.09) -- (69.00,131.71);
\path[line width=0.30mm, draw=L] (68.74,131.20) -- (73.57,120.02);
\path[line width=0.30mm, draw=L] (74.59,120.02) -- (85.77,120.02);
\path[line width=0.30mm, draw=L] (86.28,120.02) -- (91.87,131.71);
\path[line width=0.30mm, draw=L] (80.18,140.09) -- (91.87,131.45);
\path[line width=0.30mm, draw=L, fill=F] (110.68,139.59) circle (1.00mm);
\path[line width=0.30mm, draw=L, fill=F] (110.68,120.53) circle (1.00mm);
\path[line width=0.30mm, draw=L, fill=F] (122.11,129.93) circle (1.00mm);
\path[line width=0.30mm, draw=L, fill=F] (134.06,129.93) circle (1.00mm);
\path[line width=0.30mm, draw=L] (110.68,139.59) -- (110.68,120.53);
\path[line width=0.30mm, draw=L] (110.68,120.27) -- (122.37,129.93);
\path[line width=0.30mm, draw=L] (110.68,139.84) -- (121.86,129.93);
\path[line width=0.30mm, draw=L] (122.11,129.93) -- (134.06,129.93);
\path[line width=0.30mm, draw=L, fill=F] (162.01,139.59) circle (1.00mm);
\path[line width=0.30mm, draw=L, fill=F] (151.59,129.93) circle (1.00mm);
\path[line width=0.30mm, draw=L, fill=F] (162.01,119.76) circle (1.00mm);
\path[line width=0.30mm, draw=L, fill=F] (172.18,129.67) circle (1.00mm);
\path[line width=0.30mm, draw=L, fill=F] (182.09,129.67) circle (1.00mm);
\path[line width=0.30mm, draw=L] (161.76,139.84) -- (151.59,129.67);
\path[line width=0.30mm, draw=L] (151.85,129.93) -- (161.76,119.76);
\path[line width=0.30mm, draw=L] (161.76,139.84) -- (171.67,129.67);
\path[line width=0.30mm, draw=L] (171.93,129.67) -- (161.76,119.76);
\path[line width=0.30mm, draw=L] (172.18,129.67) -- (181.84,129.67);
\path[line width=0.30mm, draw=L, fill=F] (49.94,129.93) circle (1.00mm);
\path[line width=0.30mm, draw=L, fill=F] (19.70,129.93) circle (1.00mm);
\path[line width=0.30mm, draw=L, fill=F] (30.12,129.93) circle (1.00mm);
\path[line width=0.30mm, draw=L, fill=F] (40.03,129.93) circle (1.00mm);
\path[line width=0.30mm, draw=L] (19.70,129.93) -- (50.19,129.93);
\path[line width=0.30mm, draw=L] (110.42,90.03) -- (140.92,90.03);
\path[line width=0.30mm, draw=L, fill=F] (110.42,90.03) circle (1.00mm);
\path[line width=0.30mm, draw=L, fill=F] (120.84,90.03) circle (1.00mm);
\path[line width=0.30mm, draw=L, fill=F] (130.75,90.03) circle (1.00mm);
\path[line width=0.30mm, draw=L, fill=F] (140.67,90.03) circle (1.00mm);
\path[line width=0.30mm, draw=L, fill=F] (121.10,104.26) circle (1.00mm);
\path[line width=0.30mm, draw=L] (161.00,104.52) -- (161.00,90.03);
\path[line width=0.30mm, draw=L] (150.58,90.03) -- (181.07,90.03);
\path[line width=0.30mm, draw=L, fill=F] (150.58,90.03) circle (1.00mm);
\path[line width=0.30mm, draw=L, fill=F] (161.00,90.03) circle (1.00mm);
\path[line width=0.30mm, draw=L, fill=F] (170.91,90.03) circle (1.00mm);
\path[line width=0.30mm, draw=L, fill=F] (180.82,90.03) circle (1.00mm);
\path[line width=0.30mm, draw=L, fill=F] (161.00,104.52) circle (1.00mm);
\path[line width=0.30mm, draw=L] (120.84,104.01) -- (120.84,90.03);
\path[line width=0.30mm, draw=L, fill=F] (49.94,90.03) circle (1.00mm);
\path[line width=0.30mm, draw=L] (19.70,90.03) -- (50.19,90.03);
\path[line width=0.30mm, draw=L, fill=F] (19.70,90.03) circle (1.00mm);
\path[line width=0.30mm, draw=L, fill=F] (30.12,90.03) circle (1.00mm);
\path[line width=0.30mm, draw=L, fill=F] (40.03,90.03) circle (1.00mm);
\path[line width=0.30mm, draw=L, fill=F] (25.29,104.26) circle (1.00mm);
\path[line width=0.30mm, draw=L] (25.29,104.26) -- (19.70,90.03);
\path[line width=0.30mm, draw=L] (25.29,104.26) -- (30.12,90.03);
\path[line width=0.30mm, draw=L, fill=F] (131.01,104.52) circle (1.00mm);
\path[line width=0.30mm, draw=L] (131.01,104.52) -- (131.01,90.03);
\path[line width=0.30mm, draw=L] (120.84,104.01) -- (130.50,90.03);
\path[line width=0.30mm, draw=L] (131.01,104.01) -- (120.84,89.77);
\path[line width=0.30mm, draw=L] (74.08,104.01) -- (83.74,90.03);
\path[line width=0.30mm, draw=L] (63.66,90.03) -- (94.16,90.03);
\path[line width=0.30mm, draw=L, fill=F] (63.66,90.03) circle (1.00mm);
\path[line width=0.30mm, draw=L, fill=F] (74.08,90.03) circle (1.00mm);
\path[line width=0.30mm, draw=L, fill=F] (83.99,90.03) circle (1.00mm);
\path[line width=0.30mm, draw=L, fill=F] (93.90,90.03) circle (1.00mm);
\path[line width=0.30mm, draw=L, fill=F] (74.34,104.26) circle (1.00mm);
\path[line width=0.30mm, draw=L] (74.08,104.01) -- (74.08,90.03);
\path[line width=0.30mm, draw=L, fill=F] (84.25,104.52) circle (1.00mm);
\path[line width=0.30mm, draw=L] (84.25,104.52) -- (84.25,90.03);
\path[line width=0.30mm, draw=L] (74.08,104.01) -- (84.50,104.01);
\path[line width=0.30mm, draw=L] (84.25,104.01) -- (74.08,89.77);
\path[line width=0.30mm, draw=L, fill=F] (211.06,90.03) circle (1.00mm);
\path[line width=0.30mm, draw=L] (200.64,90.03) -- (231.14,90.03);
\path[line width=0.30mm, draw=L, fill=F] (200.64,90.03) circle (1.00mm);
\path[line width=0.30mm, draw=L, fill=F] (220.97,90.03) circle (1.00mm);
\path[line width=0.30mm, draw=L, fill=F] (230.89,90.03) circle (1.00mm);
\path[line width=0.30mm, draw=L, fill=F] (200.39,104.01) circle (1.00mm);
\path[line width=0.30mm, draw=L, fill=F] (210.81,123.83) circle (1.00mm);
\path[line width=0.30mm, draw=L] (210.81,138.32) -- (210.81,123.83);
\path[line width=0.30mm, draw=L] (200.39,123.83) -- (230.89,123.83);
\path[line width=0.30mm, draw=L, fill=F] (200.39,123.83) circle (1.00mm);
\path[line width=0.30mm, draw=L, fill=F] (220.72,123.83) circle (1.00mm);
\path[line width=0.30mm, draw=L, fill=F] (230.63,123.83) circle (1.00mm);
\path[line width=0.30mm, draw=L, fill=F] (210.81,138.32) circle (1.00mm);
\path[line width=0.30mm, draw=L] (210.81,138.32) -- (200.64,123.83);
\path[line width=0.30mm, draw=L] (210.81,138.06) -- (220.97,123.83);
\path[line width=0.30mm, draw=L] (200.39,103.75) -- (200.39,89.77);
\draw(31.13,112.41) node[anchor=base west]{\fontsize{11.38}{13.66}\selectfont $P_4$};
\draw(75.35,112.41) node[anchor=base west]{\fontsize{11.38}{13.66}\selectfont $C_5$};
\draw(121.61,112.41) node[anchor=base west]{\fontsize{11.38}{13.66}\selectfont $H_1$};
\draw(157.95,112.41) node[anchor=base west]{\fontsize{11.38}{13.66}\selectfont $H_2$};
\draw(211.57,112.41) node[anchor=base west]{\fontsize{11.38}{13.66}\selectfont $H_3$};
\draw(31.13,78.59) node[anchor=base west]{\fontsize{11.38}{13.66}\selectfont $H_4$};
\draw(75.35,78.59) node[anchor=base west]{\fontsize{11.38}{13.66}\selectfont $H_5$};
\draw(121.61,78.59) node[anchor=base west]{\fontsize{11.38}{13.66}\selectfont $H_6$};
\draw(157.95,78.59) node[anchor=base west]{\fontsize{11.38}{13.66}\selectfont $H_7$};
\draw(211.57,78.59) node[anchor=base west]{\fontsize{11.38}{13.66}\selectfont $H_8$};
\end{tikzpicture}%
\caption{The graphs $P_4$, $C_5$, $H_1$-$H_8$.}\label{Fig. 2}
\end{figure}
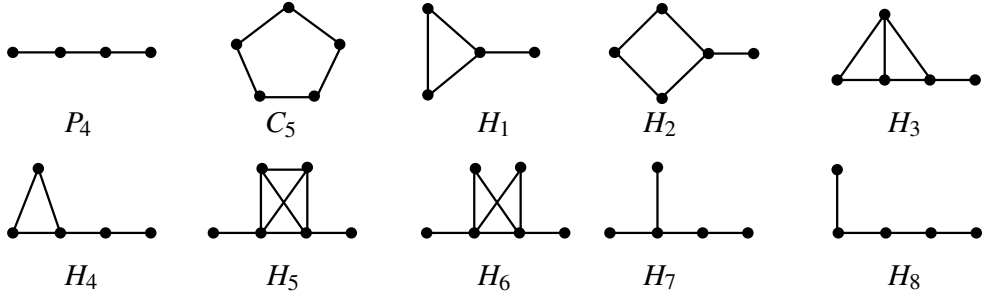

\begin{remark}\label{rek1.1}
If $(K_n,\sigma)$ is unbalanced, then the shortest negative cycle must be $(C_3,\sigma)$. Otherwise, suppose that
$(C_k,\sigma)$ is a shortest negative cycle of $(K_n,\sigma)$ and $k\geq4$.
Let $C_k=v_1v_2v_3v_4\ldots v_kv_1$. Note that $(K_n,\sigma)$ is compatible, thus $\sigma_{uv}=\sigma'_{uv}$. If $\sigma_{12}\sigma_{23}=-1$, then $\sigma_{13}=-1$, it implies that $C_{k-2}=v_1v_3v_4\ldots v_kv_1$ is a negative cycles with length $k-1$, a contradiction. If $\sigma_{12}\sigma_{23}=1$, then $\sigma_{14}=1$, it implies that $C'=v_1v_3v_4\ldots v_kv_1$ is also a negative cycles with length $k-1$, a contradiction. And $\lambda_3(D^{\pm}((K_3,\sigma)))=-2$ when $\sigma(K_3)=-1$.
\end{remark}

\begin{lemma} \cite{H.Q. Lin}.\label{lm1.2} Let $G$ be a connected graph and $D(G)$ be the distance matrix of $G$. Then $\lambda_n(D(G))=-2$ with multiplicity $n-k$ if and only if $G$ is a complete $k$-partite graph for $2\leq k\leq n-1$.
\end{lemma}

\begin{lemma}\label{lem1.3}
Let $\Sigma=(K_{n_1,n_2,\ldots,n_k},\sigma)$ be an unbalanced signed completed $k$-partite graph, where $k\geq3$. If $\Sigma$ satisfies the following conditions:
\begin{enumerate}
  \item [(i)] The paths of length two for any two nonadjacent vertices have the same sign;
  \item [(ii)] $\sigma(K_3)=-1$ for every $K_3\subset K_{n_1,n_2,\ldots,n_k},$
\end{enumerate}
then
\begin{enumerate}
  \item [(1)] $D^{max}(\Sigma)=D^{min}(\Sigma)=D^{\pm}(\Sigma)$, that is, $\Sigma$ is compatible;
  \item [(2)] $-2$ is an eigenvalue of $D^{\pm}(\Sigma)$ with multiplicity $n-k+1(i.e.,n-2)$ for $k=3$;
  \item [(3)] $-2$ is an eigenvalue of $D^{\pm}(\Sigma)$ with multiplicity $n-k$ for $k\geq4$.
\end{enumerate}
\end{lemma}
\renewcommand\proofname{\bf{Proof}}
\begin{proof}
Let $V=V_1\cup V_2\cup\ldots\cup V_k$, where $V_i$ is an independent set and $\mid V_i\mid=n_i$ for $1\leq i\leq k$. Since the diameter of $\Sigma$ is 2, $\Sigma$ is compatible by (i) immediately. We partition $V_i$ to $v_{i1}$, $V_{i2}$ and $V_{i3}$ such that the paths of length two between $v_{i1}$ and $u$ with positive sign, the paths of length two between $v_{i1}$ and $w$ with negative sign, where $u\in V_{i2}$ and $w\in V_{i3}$. Let $\mid V_{i2}\mid=s_i$ for $1\leq i\leq k$ . Then
$$D^{\pm}(\Sigma)=\left(\begin{array}{cccc}
               A_{n_1\times n_1} & A_{n_1\times n_2} & \cdots & A_{n_1\times n_k} \\
               A^T_{n_1\times n_2} & A_{n_2\times n_2} & \cdots & A_{n_2\times n_k} \\
               \vdots & \vdots & \ddots & \vdots \\
               A^T_{n_1\times n_k} & A^T_{n_2\times n_k} & \cdots & A_{n_k\times n_k} \\
             \end{array}
           \right),$$
where
$$A_{n_i\times n_i}=\left(\begin{array}{ccc}
               0 & 2J_{1\times s_i} & -2J_{1\times (n_i-s_i-1)} \\
               2J_{s_i\times1} & 2(J_{s_i\times s_i}-I_{s_i\times s_i}) & -2J_{s_i\times (n_i-s_i-1)} \\
               -2J_{(n_i-s_i-1)\times1} & -2J_{(n_i-s_i-1)\times s_i} & 2(J_{(n_i-s_i-1)\times (n_i-s_i-1)}-I_{(n_i-s_i-1)\times (n_i-s_i-1)}) \\
             \end{array}
           \right)$$
and
$$A_{n_i\times n_j}=\sigma(v_{i1}v_{j1})\left(\begin{array}{ccc}
               1 & J_{1\times s_j} & -J_{1\times (n_j-s_j-1)} \\
               J_{s_i\times1} & J_{s_i\times s_j} & -J_{s_i\times (n_j-s_j-1)} \\
               -J_{(n_i-s_i-1)\times1} & -J_{(n_i-s_i-1)\times s_j} & J_{(n_i-s_i-1)\times (n_j-s_j-1)} \\
             \end{array}
           \right).$$
Then $det(\lambda I-D^{\pm}(\Sigma))=(\lambda+2)^{n-k}f(\lambda)$, where
$$f(\lambda)=\begin{vmatrix}
                        \lambda-2n_1+2 & -\sigma(v_{11}v_{21})n_2 & \cdots & -\sigma(v_{11}v_{k1})n_k \\
                        -\sigma(v_{11}v_{21})n_1 & \lambda-2n_2+2 & \cdots & -\sigma(v_{21}v_{k1})n_k \\
                        \vdots & \vdots & \vdots & \vdots \\
                        -\sigma(v_{11}v_{k1})n_1 & -\sigma(v_{21}v_{k1})n_2 & \cdots & \lambda-2n_k+2
                        \end{vmatrix}.$$
Note that $\sigma(v_{i1}v_{j1})=-\sigma(v_{11}v_{i1})\sigma(v_{11}v_{j1})$ by (ii) for $1<i\neq j\leq k$ and $\sigma^2(v_{i1}v_{j1})=1$ for $1\leq i, j\leq k$, thus
$$f(\lambda)=\begin{vmatrix}
                        \lambda-2n_1+2 & -\sigma(v_{11}v_{21})n_2 & \cdots & -\sigma(v_{11}v_{k1})n_k \\
                        -\sigma(v_{11}v_{21})n_1 & \lambda-2n_2+2 & \cdots & \sigma(v_{11}v_{21})\sigma(v_{11}v_{k1})n_k \\
                        \vdots & \vdots & \vdots & \vdots \\
                        -\sigma(v_{11}v_{k1})n_1 & \sigma(v_{11}v_{21})\sigma(v_{11}v_{31})n_2 & \cdots & \lambda-2n_k+2
                        \end{vmatrix}.$$
And
\begin{align*}
f(-2)&=\begin{vmatrix}
                        -2n_1 & -\sigma(v_{11}v_{21})n_2 & \cdots & -\sigma(v_{11}v_{k1})n_k \\
                        -\sigma(v_{11}v_{21})n_1 & -2n_2 & \cdots & \sigma(v_{11}v_{21})\sigma(v_{11}v_{k1})n_k \\
                        \vdots & \vdots & \vdots & \vdots \\
                        -\sigma(v_{11}v_{k1})n_1 & \sigma(v_{11}v_{21})\sigma(v_{11}v_{31})n_2 & \cdots & -2n_k
                        \end{vmatrix}\\
\ &=\Pi_{i=1}^kn_k \begin{vmatrix}
                        -2 & -1 & \cdots & -1 \\
                        -1 & -2 & \cdots & 1 \\
                        \vdots & \vdots & \vdots & \vdots \\
                        -1 & 1 & \cdots & -2
                        \end{vmatrix} \\
\ &=\Pi_{i=1}^kn_k (-3)^{k-1}(k-3),
\end{align*}
$f(-2)=0$ if and only if $k=3$ and this shows that (3) is correct.

If $k=3$, then $det(\lambda I-D^{\pm}((K_{n_1,n_2,,n_3},\sigma))=(\lambda+2)^{n-3}g(\lambda)$, where
\begin{align*}
g(\lambda)&=\begin{vmatrix}
                      \lambda-2n_1+2 & -\sigma(v_{11}v_{21})n_2 & -\sigma(v_{11}v_{31})n_3 \\
                        -\sigma(v_{11}v_{21})n_1 & \lambda-2n_2+2 & \sigma(v_{11}v_{21})\sigma(v_{11}v_{31})n_3 \\
                        -\sigma(v_{11}v_{31})n_1 & \sigma(v_{11}v_{21})\sigma(v_{11}v_{31})n_2 & \lambda-2n_3+2
                        \end{vmatrix}\\
\ &=(\lambda+2)h(\lambda),
\end{align*}
and $h(\lambda)={\lambda}^{2}+ \left( -2\,n+4 \right) \lambda+3\,n_{{1}}n_{{2}}+3\,n_{
{1}}n_{{3}}+3\,n_{{2}}n_{{3}}-4\,n+4$, note that $h(-2)=3(n_1n_2+n_1n_3+n_2n_3)\neq0$. It means that the multiplicity of $-2$ is $n-2$ and (2) holds.                This completes the proof.
\end{proof}

\begin{lemma}\cite{S.K. Hameed}\label{lem1.1}
The following properties of a signed graph $\Sigma$ are equivalent.
\begin{enumerate}
  \item [(i)] $\Sigma$ is balanced;
  \item [(ii)] $D^{max}(\Sigma)$ is cospectral with $D(G)$;
  \item [(iii)] $D^{min}(\Sigma)$ is cospectral with $D(G)$.
\end{enumerate}
\end{lemma}

The main outcomes of our investigation are as follows.

\begin{thm}\label{thm12}
Let $\Sigma=(G,\sigma)$ be a connected signed graph and $D^{max}$ and $D^{min}$ be two signed distance matrices of $\Sigma$. Then
\begin{enumerate}
   \item[(I)] $\lambda_n(D^{max}(\Sigma))=\lambda_n(D^{min}(\Sigma))=-1$ if and only if $\Sigma$ is a balanced signed completed graph $K_n$;
   \item[(II)] There does not exist a signed graph $\Sigma$ of order $n$ such that $\{\lambda_n(D^{max}(\Sigma)),\lambda_n(D^{min}(\Sigma))\}\in(-2,-1)$;
   \item[(III)] $\lambda_n(D^{max}(\Sigma))=\lambda_n(D^{min}(\Sigma))=-2$ with multiplicity $n-k$ if and only if $\Sigma$ is a balanced signed complete $k$-partite graph with $2\leq k\leq n-1$.
 \end{enumerate}
\end{thm}

\begin{proof}
(I). Sufficiency is clearly established. If diameter $d\geq2$, then both $\lambda_n(D^{max})$ and $\lambda_n(D^{min})$ are less than or equal to $-2$ by Lemma \ref{lem1.11}, a contradiction. Thus, $d=1$, i.e., $G\cong K_n$ and the proof is completed by Lemma \ref{lem1.1} and Remark \ref{rek1.1}. And there does not exist a signed graph $\Sigma$ of order $n$ such that $\{\lambda_n(D^{max}(\Sigma)),\lambda_n(D^{min}(\Sigma))\}\in(-2,-1)$ clearly, (II) is correct.

(III). The sufficiency holds immediately by Lemmas \ref{lem1.1} and \ref{lm1.2}.
Conversely, we first claim that diameter $d\leq2$. Otherwise, $(P_4,\sigma)$ is an induced subgraph of $\Sigma$, and both $\lambda_n(D^{max}(\Sigma))$ and $\lambda_n(D^{min}(\Sigma))$ less than $-2$ by Lemma \ref{lm1}. Now, we derive the proof into the following cases.

\noindent{\bf{$\underline{\mbox{Case 1. }}$}} $\Sigma$ is a signed tree.
Obviously, $\Sigma$ is compatible, that is, $D^{max}(\Sigma)=D^{min}(\Sigma)=D^{\pm}(\Sigma)$.
Thus $\Sigma=(K_{1,n-1},\sigma)$ and the conclusion is valid by Lemmas \ref{lem1.1} and \ref{lm1.2}.

\noindent{\bf{$\underline{\mbox{Case 2.}}$}} $\Sigma$ contains signed cycles.
Since the least eigenvalues of the principle submatrices indexed by $V(P_4)$ or $V(C_5)$ are always less than $-2$, the girth $g\leq4$.

\noindent{\bf{$\underline{\mbox{Subcase 2.1.}}$}}\label{case2} $g=3$.
If $\Sigma$ is an unbalanced complete graph, the the shortest negative cycle must be $C_3$ by Remark \ref{rek1.1}. If $n\geq4$, then $(K_4,\sigma)\subset \Sigma$ and $D^{max}(\Sigma)=D^{min}(\Sigma)=D^{\pm}(\Sigma)$ must contain one of the following matrices as a principle submatrix,
$$\left(
             \begin{array}{ccccc}
               0 & -1 & 1 & i \\
               -1 & 0 & 1 & j \\
               1 & 1 & 0 & k \\
               i & j & k & 0  \\
             \end{array}
           \right)~~~~
 \left(
             \begin{array}{ccccc}
               0 & -1 & -1 & i \\
               -1 & 0 & -1 & j \\
               -1 & -1 & 0 & k \\
               i & j & k & 0  \\
             \end{array}
           \right),$$
where $\mid i\mid=\mid j\mid=\mid k\mid=1$. It is a pity that the least eigenvalues of the above matrices are always less than $-2$, a contradiction. Thus, $\Sigma$ is an unbalanced $K_3$ and its eigenvalues are $-2$ and $1$ with multiplicity $2$, a contradiction.

Thus $\Sigma$ is not an unbalanced complete graph. Since the least eigenvalue of the principle submatrix indexed by $V(H_1)$ is always less than $-2$ and $n\geq4$, then $V(\Sigma)$ can be partitioned into four parts $V_1, V_2, V_3$ and $V_4$, where $V_1=\{u|uv_2,uv_3\in E(\Sigma)\}$, $V_2=\{u|uv_1,uv_3\in E(\Sigma)\}$, $V_3=\{u|uv_1,uv_2\in E(\Sigma)\}$ and $V_4=\{u|uv_1, uv_2,uv_3\in E(\Sigma)\}$. Clearly, $v_1\in V_1$, $v_2\in V_2$ and $v_3\in V_3$. Meanwhile, $V_i$ is an independent set and every vertex of $V_j$ is adjacent to every vertex of $V_k$ for $1\leq i\leq3$ and $1\leq j\neq k\leq4$. We partition $V_4$ to independent sets $V_{41}\cup\cdots \cup V_{4t}$, then $uv\in E$ for any $u\in V_{4i}$ and $v\in V_{4j}$ for $1\leq i\neq j\leq t$ and $t\geq2$ through the same methods in \cite{H.Q. Lin}. Therefore, $\Sigma$ is a signed completed multipartite graph $(K_{n_1,n_2,\ldots,n_k},\sigma)$. If $\Sigma$ is balanced, then the proof is completed by Lemmas \ref{lem1.1} and \ref{lm1.2}. Next, let's assume $\Sigma$ is unbalanced. Then $\Sigma$ must be an unbalanced completed $3$-partite graph since $\Sigma$ cannot contain an unbalanced $K_4$ as an induced subgraph. Let $V=V_1\cup V_2\cup V_3$, where $\mid V_i\mid= n_i\geq1$ and $n=n_1+n_2+n_3\geq5$.

\begin{claim}\label{claim1}
The paths of length two for any two nonadjacent vertices have the same sign.
\end{claim}
Otherwise, $D^{max}(\Sigma)$ must contain one of $\{A_6,A_7\}$ as a principle submatrix,
$$A_6=\left(
             \begin{array}{ccccc}
               0 & 1 & 2 & 1 \\
               1 & 0 & -1 & i \\
               2 & -1 & 0 & 1 \\
               1 & i & 1 & 0  \\
             \end{array}
           \right)~~~\mbox{or}~~~
 A_7=\left(
             \begin{array}{ccccc}
               0 & 1 & 2 & -1 \\
               1 & 0 & -1 & i \\
               2 & -1 & 0 & -1 \\
               -1 & i & -1 & 0  \\
             \end{array}
           \right)$$
and $D^{min}(\Sigma)$ must contain one of $\{A_8,A_9\}$ as a principle submatrix,
$$A_8=\left(
             \begin{array}{ccccc}
               0 & 1 & -2 & 1 \\
               1 & 0 & -1 & j \\
               -2 & -1 & 0 & 1 \\
               1 & j & 1 & 0  \\
             \end{array}
           \right)~~\mbox{or}~~
 A_9\left(
             \begin{array}{ccccc}
               0 & 1 & -2 & -1 \\
               1 & 0 & -1 & j \\
               -2 & -1 & 0 & -1 \\
               -1 & j & -1 & 0  \\
             \end{array}
           \right),$$
where $i=\pm1,2$, $j=\pm1,-2$. Yet, sad to say, the least eigenvalues of the above matrices are less than $-2$, a contradiction.

\begin{claim}\label{claim2}
$(K_{n_i,n_j},\sigma)$ is balanced for $n_i,n_j\geq2$, $1\leq i\neq j\leq 3$.
\end{claim}

Otherwise, assume that $C_{2k}=v_1v_2\ldots v_{2k}$ is a shortest negative cycle, then $k\geq3$ since $\sigma(C_4)=1$ for any $(C_4,\sigma)\subset \Sigma$ by Claim \ref{claim1}. Without loss of generality, assume that $\sigma_{12}=1$ and $\sigma_{1,2k}=-1$. Let $C'=v_1v_2v_{2k-1}v_{2k}v_1$ and $C''=v_2v_3\ldots v_{2k-2}v_{2k-1}v_2$, then $\sigma(C'')=-1$ since $\sigma(C')=1$, but the length of $C''$ is $2k-2$, a contradiction.

\begin{claim}\label{claim3}
The length of shortest negative cycle must be three.
\end{claim}
Otherwise, assume that $(C_{k},\sigma)$ is a shortest negative cycle, then $k\geq5$ by Claim \ref{claim1}. We assume that $\sigma_{12}=1$ and $\sigma_{1k}=-1$ at first. If $v_2v_k\in E$, then $\sigma_{2k}=-1$. We can construct a new negative cycle $C'=v_2\ldots v_kv_2$, but the length of $C'$ is $k-1$, a contradiction. If $v_2v_k\notin E$, then $v_2v_{k-1}\in E$ and $\sigma_{2,k-1}\sigma_{k,k-1}=-1$ by Claim \ref{claim1}. Constructing a new negative cycle $C''=v_2\ldots v_{k-1}v_2$ with length $k-2$, a contradiction. We can get the similar contradiction when $\sigma(v_1v_2)=\sigma(v_1v_{k})=-1$.
\begin{claim}\label{claim4}
$\sigma(K_3)=-1$ for every $K_3\subset K_{n_1,n_2,n_3}$.
\end{claim}
We will prove the claim by using reduction to absurdity. Note that any two triangles with a common edge must have the same sign by Claim \ref{claim1}. Suppose that $C'=v_1v_2v_3v_1$, $C''=v'_1v'_2v'_3v'_1$, where $\sigma(C')\sigma(C'')=-1$ and $v_i, v'_i\in V_i$ for $1\leq i\leq3$. Let $C'''=v_1v'_2v_3v_1$. If $v_1=v'_1$, $v_2\neq v'_2$ and $v_3\neq v'_3$, then $\sigma(C')\sigma(C''')=1$ and $\sigma(C'')\sigma(C''')=1$, it is impossible. And this situation shows that any two triangles with a common vertex also have the same sign. If $v_1\neq v'_1$, $v_2\neq v'_2$ and $v_3\neq v'_3$, then $\sigma(C')\sigma(C''')=1$ and $\sigma(C'')\sigma(C''')=1$, it is impossible. Therefore, Claim \ref{claim4} holds by Claim \ref{claim3}.

By Claims \ref{claim1}-\ref{claim4}, we discover that $\Sigma$ is an unbalanced completed $3$-partite graph which satisfies the conditions in Lemma \ref{lem1.3}. However, the multiplicity of $-2$ is $n-2$, a contradiction.

\noindent{\bf{$\underline{\mbox{Subcase 2.2.}}$}} $g=4$.
Since the least eigenvalue of the principle submatrix indexed by $V(H_2)$ is always less than $-2$, we can get $\Sigma$ is a signed complete bipartite graph through the same methods in \cite{H.Q. Lin}. And then $\Sigma$ is balanced by Claim \ref{claim2} of Subcase 2.1 immediately.
\end{proof}

\end{document}